\documentclass[11pt]{article}
\usepackage{amsthm}
\usepackage{amsmath}
\usepackage{amsfonts}
\usepackage{amssymb}
\usepackage[all, cmtip]{xy}
\usepackage{amscd}
\usepackage{hyperref}
\usepackage{geometry}
\newtheorem{thm}{Theorem}[section]
\newtheorem{remark}[thm]{Remark}
\newtheorem{lemma}[thm]{Lemma}
\newtheorem{proposition}[thm]{Proposition}

\newtheorem{definition}[thm]{Definition}
\newtheorem{theorem}[thm]{Theorem}
\newtheorem{corollary}[thm]{Corollary}

\newtheorem{example}{Example}

\newcommand{\mZ}{\mathbb{Z}}
\newcommand{\mL}{\mathbb{L}}
\newcommand{\mE}{\mathbb{E}}
\newcommand{\mr}[1]{\ensuremath{\mathrm{#1}}}

\newcommand{\sA}{\s{A}}
\newcommand{\sU}{\s{U}}
\newcommand{\sH}{\s{H}}

\newcommand{\sL}{\s{L}}
\newcommand{\sV}{\s{V}}
\newcommand{\sO}{\s{O}}
\newcommand{\sF}{\s{F}}
\newcommand{\sG}{\s{G}}

\newcommand{\sP}{\s{P}}
\newcommand{\sM}{\s{M}}
\newcommand{\sE}{\s{E}}
\newcommand{\sC}{\s{C}}
\newcommand{\s}[1]{\ensuremath{\mathcal{#1}}}
\newcommand{\spec}{\operatorname{Spec}}
\newcommand{\bderive}[1]{\ensuremath{\mathrm{D}^{\mathrm{b}}(#1)}}
\newcommand{\sdHom}{R\mathcal{H}om}
\newcommand{\sMab}{\ensuremath{\sM^{[a,b]}_{\bderive{X}}}}

\newcommand{\sMtor}{\ensuremath{\sM_{\bderive{X}}}}
\newcommand{\bderivefl}[1]{\ensuremath{\mathrm{D}^{\mathrm{b}}_{fl,f}(#1)}}
\newcommand{\bderiveflpi}[1]{\ensuremath{\mathrm{D}^{\mathrm{b}}_{fl,\pi}(#1)}}
\newcommand{\bderiveflab}[3]{\ensuremath{\mathrm{D}^{[#2, #3]}_{fl, \pi}(#1)}} 
\newcommand{\Vect}{\ensuremath{\mr{\bf{Vect}}}}
\newcommand{\QC}{\ensuremath{\operatorname{\mr{QCoh}}}}
\newcommand{\Map}{\ensuremath{Map}}
\newcommand{\kalg}{k\textrm{-Alg}_{\Delta}}

\newcommand{\Morph}{M\mr{orph}}
\newcommand{\Aalg}{A\textrm{-Alg}}
\newcommand{\Amod}{A\textrm{-mod}}
\newcommand{\kmod}{k\textrm{-mod}_{\Delta}}
\newcommand{\cofib}{\ensuremath{\operatorname{\mr{cofib}}}}
\newcommand{\fib}{\ensuremath{\operatorname{\mr{fib}}}}
\newcommand{\sset}{\ensuremath{\mr{Set}_\Delta}}
\newcommand{\Perf}{\ensuremath{Perf}}

\newcommand{\rank}{\ensuremath{\operatorname{rank}}}

\newcommand{\QCF}{\ensuremath{\widetilde{\QC(X)}}}
\newcommand{\sMabone}{\ensuremath{\sM^{[a,b+1]}_{\bderive{X}}}}
\newcommand{\mV}{\mathbb{V}}
\newcommand{\mI}{\mathbb{I}}
\newcommand{\pscoh}{\mr{PsCoh}}
\newcommand{\Kalg}{K\textrm{-Alg}}
\newcommand{\Kmod}{K\textrm{-mod}}
\newcommand{\kAlg}{k\textrm{-Alg}}
\newcommand{\Balg}{B\textrm{-Alg}}

\author{Parker E. Lowrey}
\title{The moduli stack and motivic Hall algebra for the bounded derived category.}

\begin{document}

\maketitle
\abstract{
We give an alternate formulation of pseudo-coherence over an arbitrary derived stack $X$.  The full subcategory of pseudo-coherent objects forms a stable sub-$\infty$-category of the derived category associated to $X$.  Using relative Tor-amplitude we define a derived stack $\sMtor$ classifying pseudo-coherent objects.  For reasonable base schemes, this classifies the bounded derived category. In the case that $X$ is a projective derived scheme flat over the base, we show $\sMtor$ is locally geometric and locally of almost finite type.  Using this result, we prove the existence of a derived motivic Hall algebra associated to $X$.  
}






\section{Introduction}

The derived category was first introduced as a tool to study derived functors (and their compositions); this framework was highly effective in unifying topological and algebraic invariants under the same guise, e.g., Tor, singular cohomology.  Since its inception, the derived category has moved from being a tool, to being an invariant itself.  As an invariant it is still largely mysterious.  Technology for extracting more concise invariants from the derived category includes constructing ``universal'' invariants like K-theory and more local invariants via performing intersection theory on moduli of nicely behaving objects, e.g, DT and PT invariants.
This latter technology has been restricted to objects with no negative $\mr{Ext}$ groups.  A first step to extending these techniques to a broader class of objects is to create nicely behaving moduli spaces that incorporate objects with non-trivial negative $\mr{Ext}$ groups.  If $X$ is smooth To\"en and Vaqui\'e \cite{dg_moduli} show that one can indeed form a reasonable geometric space classifying the bounded derived category.  This paper shows that the same is true for any scheme satisfying projectivity conditions (and should be true for proper varieties as well).

The recent developments that allow us to address objects with higher automorphisms come from applications of homotopical algebra to sheaf theory.  These connections were originally established by homotopy theorists studying the lack of descent in algebraic K-theory. Starting with Jardine \cite{Jardine} and culminating in the work of To\"en, Vessozi, and Lurie \cite{HAGII, topoi, DAGV} 
the context which to build our moduli spaces is now well developed.  In particular, we will be employing  $\infty$-categories to work with homotopy coherence in absence of a model structure on the derived categories of interest and higher derived stacks to describe our moduli spaces.  The higher stacks will be used to describe the higher automorphisms of our objects and the ``derived'' data will ensure we incorporate the natural infinitesimal information included in the derived category.  




To be more specific about the problem at hand, suppose $X$ is a smooth quasi-separated quasi-compact scheme.  One has no hope of finding a well behaved moduli stack that classifies unbounded chain complexes of quasi-coherent sheaves.  To get a space of finite type, one needs to impose finiteness conditions on the objects classified.  As shown in \cite{dg_moduli} and \cite{Pranav}, if one imposes dualizability (a strong finiteness condition), then the moduli problem is indeed a locally geometric locally finite type stack, which is the nicest possible result that one can (and should) expect. In the case that $X$ is smooth, dualizability is equivalent to most other reasonable finiteness conditions, e.g., restricting to bounded finite type modules, compactness.  This statement is no different than the common correspondence between homological and cohomological data on smooth manifolds.  When $X$ is singular, certain contexts yield dualizability insufficient to study $X$.  For example, skyscraper sheaves supported on the singular locus will not be dualizable.

The correct notion for singular $X$ is an extension of ``homological finiteness'' to the derived stack setting via the notion of {\it{pseudo-coherent}} objects.  Classically, these notions were developed by Illusie in \cite{SGA6} to fix stability issues with the bounded derived category for arbitrarily pathological topoi. If one assumes $X$ is a Noetherian scheme, the category of pseudo-coherent objects is equivalent to the category of bounded above chain complexes of vector bundles.  Brutal truncations of these objects are then perfect.  This aspect was central to the nice results obtained by Thomason \cite{TT}.  In the derived setting pseudo-coherence also goes by the moniker of ``almost perfect'' and appears in Lurie's thesis \cite{lurie_thesis}.  We give an alternate (equivalent) definition that allows for the easy porting of the works of \cite{SGA6} and \cite{TT}.  The motivation for this difference is the fact that in the derived setting we no longer have easily describable brutal truncations;  simplicially these correspond to skeleton functors, which if carried out as is will not preserve module structures.   Therefore, we force a brutal truncation locally via the idea of filtered objects \cite{DAGI}.  Porting over many of the proofs from \cite{SGA6} and \cite{TT} is then straightforward.

Depending on how close $X$ is to a reasonable scheme, one can simplify this definition to a nice categorical concept.  The importance of a nice categorical definition lies in the ease at which it allows one to create and study functorial constructions.  For example, the aforementioned works on the moduli of perfect objects necessarily depend on the well behaved categorical concept of compactness.  We have the following derived analog of \cite[2.4.2]{TT} when $X$ is perfect (a broad class of nicely behaving derived stacks, as defined in \cite{BZ}).
\begin{proposition}
  If $X$ is perfect and quasi-compact, then $F \in \QC(X)$ is pseudo-coherent if and only if the $n$th truncation is compact with respect to $n$-truncated colimits.
\end{proposition}
\noindent If $X$ is not perfect, then there is no expectation that dualizable objects are compact and the resulting moduli space will miss important objects, e.g., the structure sheaf.

Once we introduce and briefly discuss the definition of pseudo-coherence we are able to define the main object of study in this paper: the moduli stack of pseudo-coherent objects with finite relative Tor-amplitude (over the base).  Our main theorem is 
\begin{theorem}
  Let $X$ be a derived projective variety flat over a commutative simplicial ring $k$.  Then the moduli stack classifying the bounded derived category, $\sM_{\bderive{X}}$, is a locally geometric stack almost of finite type.
\end{theorem}
\noindent Almost finite type means finite type if one restricts to the category of $n$-truncated simplicial $k$-algebras for any $n \in \mathbb{N}$.  We remark that this result should hold over connective $\mE_\infty$-ring spectrum as well and will probably show up in future revisions of this paper.
This theorem covers all projective schemes over an arbitrary base field. In cases of overlap with \cite{dg_moduli, Pranav}, the moduli stack of compact/perfect objects studied in those works will be a substack of our $\sM_{\bderive{X}}$; it will be an equivalence if $X$ is smooth.  

In the case that $k$ satisfies additional finiteness conditions, e.g., $\pi_0(X)$ is Noetherian with finite-type homotopy groups and $n$-truncated for some $n$, a pseudo-coherent object with bounded relative Tor-amplitude is the same as a chain complex of quasi-coherent sheaves with a finite number of non-zero cohomology groups, all finite type.  It is known that these objects can be considered as compact objects in the category Ind-coh$(X)$.  However, this does not allow the application of \cite{Pranav} since this category is not locally compact.  The lack of local compactness is the primary reason our stacks are almost finite type, not finite type.

The general ideas of how we prove the moduli functor classifying pseudo-coherent objects is a geometric (algebraic) stack is fairly standard.  Generically speaking, there exists two methods to prove that a moduli functor is a stack.  The first is by explicitly verifying the diagonal is representable and constructing a smooth atlas.  The second involves verifying the Artin conditions (or the Artin-Lurie conditions in our case).  We have chosen the first due to the explicit understanding it can give to the moduli problem.  The drawback to this approach lies in the limitation of its application.   This is already evident when working objects that ``should'' belong to abelian subcategories: there is overlap between the cases handled in \cite{lieblich} and \cite{algebraic_stacks}, the former covering a far larger class of morphisms.


Using To\"en's Grothendieck ring of higher $n$-stacks and a similiar process given in \cite[Section 3.3]{toen_overview}, we use the above theorem to show
\begin{theorem}
  Let $X$ be a projective scheme flat over a commutative ring $k$. Then there is a motivic Hall algebra $\sH(\sMtor)$ associated to $X$.
\end{theorem}
\noindent This extends the well known motivic Hall algebra structure as defined in Joyce \cite{Joyce_II} and Bridgeland \cite{Bridgeland_Hall}.  In the case that $X$ is smooth, this Hall algebra should be comparable to that defined in \cite{KS}.  

With the results of this project in hand, it is clear that it can be extended to a larger class of morphisms using the Artin-Lurie conditions.  The larger class has yet to be determined, but should be settled once the various notions of properness in derived algebraic geometry are studied.  Additional future work will include studying various compactifications of the moduli of perfect objects in the moduli of pseudo-coherent, akin to \cite{Oda}, cutting various substacks out of $\sMtor$ either via stability conditions or cohomological strictifying maps (that ideally would alter the Hall numbers).

\subsection{Acknowledgments}
The author would like to thank David Ben-Zvi for helpful suggestions on both the mathematics and the format, Bertrand To\"en for helpful suggestions in the early stages of this project, and Rick Jardine for answering various questions regarding simplicial presheaves.

\subsection{Notation}
\label{notation}
When convenient, for $X$ a stack, we will refer to $X \times \spec A$ as $X[A]$, and per the usual abuse, we often times will not distinguish between an element of $Q\in F(A)$ and the resulting morphism $Q: \spec A \to F$.

Given a (left, right, coCartesian, Cartesian, biCartesian) fibration $\sC$, we denote a choice of associated functor as $F_{\sC}$.  Lastly, we will let $\iota_X$ denote the natural counit of adjunction $\iota_X: t_0(X) \to X$.

\section{The parlance of derived algebraic geometry.}
\label{parlance}
Throughout this paper we will assume the reader is familiar with the notion of $\infty$-categories, model categories, and the various types of fibrations of $\infty$-categories studied in \cite{topoi}, e.g., right fibrations, left fibrations, Cartesian fibrations, and coCartesian fibrations. For each notion of fibration, there are straightening functors that convert the fibrations  to (sometimes contravariant) functors on the simplicial category associated to the fibration base.  The advantages to using fibrations will become clear in the latter half of this paper.  

As discussed in the introduction, the proper framework with which to treat moduli ``spaces'' of objects in a derived category is the notion of a geometric derived stack.  Good overviews of this subject include the introduction of \cite{BZ}.  The adjective ``derived'' means that one extends the basic building blocks of algebraic geometry, i.e., affine schemes, to the opposite category of one of the following choices 
\begin{itemize}
\item negatively graded commutative rings
\item simplicial commutative rings
\item connective $\mathbb{E}_\infty$ ring spectra
\end{itemize}
The last notion is the most natural from a homotopy theory viewpoint, but possibly not as accessible to the algebraic geometer.  Fortunately, if one wants to work over characteristic $0$, all three categories are equivalent (as simplicially enriched categories). 

Let $k$ be a unital (discrete) commutative ring, 
 and denote by $\kalg$ the $\infty$-category of simplicial commutative $k$-algebras 
With this motivation, we will be working with $D^-$-stacks (also referred to as derived stacks) over $\kalg$ with the \'etale or $fppf$ topology.  The category of $D^-$-stacks is a subcategory of the $\infty$-category $\mr{Fun}(\kalg, \s{S})$ where $\s{S}$ is the $\infty$-category of spaces (Kan complexes).  Let $A \in \kalg$,  $\spec A$ will denote the functor
\begin{align*}
  \kalg & \rightarrow \s{S} \\ B &\rightarrow \Map^L(A, B)
\end{align*}
and $\Aalg$ will denote the natural left fibration $(\kalg)_{A/}$ over $\kalg$.  Under the aforementioned conversion between fibrations and functors, $\spec A$ and $\Aalg$ are equivalent.  Associated to $D^-$-stack $X$ we can restrict to the full subcategory of discrete commutative $k$-algebras.  Following the notation in \cite{HAGII}, we denote this restriction as $t_0(X)$.  By adjointness, there is a natural morphism $\iota_X: t_0(X) \rightarrow X$.

A Cartesian morphism (coCartesian morphism) over a simplicial set $S$ is a fibrant-cofibrant object in the category of marked simplicial sets $(\sset^+)_{/S}$ with the Cartesian (coCartesian) model structure.  The functor $\sharp: (\sset)_{/S} \rightarrow (\sset^+)_{/S}$ has a right adjoint $\sM$ that makes $(\sharp, \sM)$ a Quillen adjunction. 
The counit of this adjunction is the classifying fibration associated to the Cartesian fibration.  As a matter of notation, if $P \rightarrow S$ is the Cartesian fibration, we denote the classifying fibration as $\sM_{\sP}$ (similar notation applies to coCartesian morphisms).  Recall that the category $(\sset)_{/S}$ with the covariant model structure is Quillen equivalent to the category $\mr{Fun}(\mathfrak{C} S, \sset)$ with the projective model structure.  Given $\sC \in (\sset)_{/S}$, we will denote a choice of fibrant-cofibrant functor associated to it as $F_{\sC}$ (well defined up to equivalence).

\subsection{The $\infty$-category of quasi-coherent sheaves.}

Given $A \in \kalg$, one has the stable symmetric monodial $\infty$-category of modules $\QC(A)$.  Similar to above, if $\textrm{char } k = 0$ this is equivalent to the category of dg-modules over the commutative dg-k-algebra $N(A)$.  This category (as we've written it) comes equipped with a t-structure generated by $A$ with its natural $A$-module structure \cite[Proposition 16.1]{DAGI}. We denote the heart of this t-structure as $\QC^0(A)$, and more generally $\QC^{[a,b]}(X)$, the full $\infty$-subcategory consisting of objects with cohomology (coming from the t-structure) vanishing outside of $[a,b]$.  We also have ``good truncation'' functors $\tau^a: \QC(A) \rightarrow \QC^{[a, \infty)}(A)$ and $\tau_a: \QC(A) \rightarrow \QC^{(-\infty, a]}(A)$. Due to the topological influence, we sometimes write $\pi_n(M)$ for $H^{-n}(M)$.  We denote the full sub-$\infty$-category $\QC^{\leq 0}(A)$ by $A$-mod. It is well known that $\QC(A)$ is the stabilization of this t-structure.    We note that there exists a forgetful functor from $A$-mod to the $\infty$-category $\kmod$.  With these conventions, it is clear that for $K \in \QC^0(A)$, $K$ is naturally a $\pi_0(A)$-module.  Further, $\pi_k(M)$ are $\pi_0(A)$-modules for any $M \in \QC(A)$.  In addition to $A$-mod, the sub-$\infty$-category $\Perf(A)$ will also be needed.  This is the smallest stable $\infty$-category of $\QC(A)$ containing $A$. 

As reviewed in \cite{BZ}, these definitions make sense for a large class of geometric stacks, which include derived schemes. For $X$ a $D^-$-stack, we will denote the associated $\infty$-category of modules by $\QC(X)$. For a general $D^-$-stack, the definition of $\QC(X)$ can be quite complicated.  However if $X$ is a geometric $D^-$-stack with affine diagonal, the description becomes much easier: it is $\varprojlim_{\Delta} \QC(U_n)$, where $n \rightarrow U_n$ is the simplicial affine scheme obtained via the \v{C}ech nerve associated to any atlas. For $f: A \rightarrow B \in \kalg$ and $M \in \QC^{(-\infty, 0]}(A)$,  $M\otimes_A B \in \QC^{(-\infty, 0]}(B)$.  This shows that there is a natural t-structure on $\QC(X)$.  We can extend the definition of a perfect complex on $X$ as an object $\s{P} \in \QC(X)$ such that its pullback to any affine is perfect.  The full stable subcategory of perfect complexes will be denoted by $\Perf(X)$. 

Since all of our categories are $\infty$-categories all functors in this paper are derived. For instance, for $A, B \in \kalg$, $A\otimes B$ is defined to be $A \otimes^\mL B$ in the natural model structure on $\kalg$.  Unless we explicitly state that limit or colimit is underived, assume all limits and colimits are homotopy limits and homotopy colimits.


Throughout the text below, many times we pullback to the underlying closed discrete subscheme.  The appropriate way of thinking about this is by \cite{HAGII}: $\kmod$ and $\kalg$ behave exactly like the underived counterparts.  In this manner, one should treat this $\infty$-category more as an abelian category rather than attempting to understand it in its stable envelope.  One gets lured into thinking of the homotopy groups as stable homotopy groups and this leads to incorrect intuition.

To define pseudo-coherent objects we will work with filtered objects in stable $\infty$-categories.  If we denote (by abuse of notation) $\mZ$ as the natural category associated to the integers, a filtered object of $\QC(X)$ is just a functor $\mr{Fun}(N\mZ, \QC(X))$.  A more complete discussion can be found in \cite[Section 11]{DAGI}.   

\section{Pseudo-coherence}

When moving from the discrete ``classical" world to the simplicial ``derived" world, the bounded derived category of coherent sheaves becomes difficult to define.  In particular, if one tries directly to use the t-structure on $A$-mod to define the category, then $A$ itself, a compact/dualizable object will not be in the category.  In fact, it will contain very little ``perfect" objects.  To fix this we use ideas of Illusie \cite{SGA6} and Thomason \cite{TT}.  It was already known in \cite{SGA6}, that coherent sheaves and chain complexes of coherent sheaves have many drawbacks if the scheme doesn't satisfy nice finiteness properties.  In particular, any reasonable definition should contain $Perf(X)$.  To fix these, Illusie defined and studied pseudo-coherent sheaves over a general ringed topos. 

 A direct port of these definition to the derived setting is not possible since vector bundles on an affine derived scheme $\spec A$ will not be discrete unless $A$ is discrete.  Thus, grading our objects by a t-structure becomes the wrong notion.  At play here is a break between algebraic and topological invariants when $A \in \kalg$ is not discrete.  For example, if we think of modules in terms of generators and relations, free algebras are our basic building blocks.  While topologically, discrete spaces are our basic building blocks; these notions depart when $A$ is not discrete.

Any reasonable ``grading" of $\QC(X)$ will ensure that $\sO_X$ has grade zero.  For this reason, Tor-amplitude will record our algebraic invariants and the natural t-structure will record the topological invariants. From a homotopical perspective, if $A$ is discrete the bad truncation $\sigma^{\geq m}$ corresponds to the $m$th skeleton.  Thus if $F \in \QC(A)$ is  $n$-pseudo-coherent and  $m < n$, the difference between the $m$ and $m+1$ skeleton of $F$ is a dualizable object.  To properly carry this out, one should have an idea of a ``relative" skeleton that records things in cellular length.  

 

\begin{definition}
Let $\sF \in \QC(X)$.  Then $\sF$ has Tor-amplitude of $[a,b]$ if for all $M \in \QC^0(X)$, $\sF\otimes M \in QC^{[a,b]}(X)$. It has locally finite Tor amplitude if it has finite Tor-amplitude on an open cover of $X$.
\end{definition}
\begin{remark}
This definition is equivalent to the standard definition since our topoi have enough points.
\end{remark}

For $X = \spec A$, we have the following properties from \cite{dg_moduli}:
\begin{proposition}
\label{absolute_tor} \mbox{}
\begin{enumerate}
\item $M \in \Amod$ has Tor-amplitude of 0 if and only if it is flat, i.e. $\otimes M$ commutes with finite limits in $\Amod$.
\item $M \in \Amod$ has Tor-amplitude of $[a,b]$ if and only if $M \otimes \pi_0(A)$ has Tor-amplitude $[a,b]$ as a $\pi_0(A)$ module.
\item If $M \to N \to P$ forms a distinguished triangle with two of the elements having finite Tor-amplitude, then so does the third.
 \end{enumerate}
\end{proposition}
\begin{proof}
{\it 1)} It is clear that flat implies $M$ has Tor-amplitude of $0$.  For the other direction, just note that if
\begin{displaymath}
\xymatrix{ A \ar[r] \ar[d] & B \ar[d] \\ 0 \ar[r] & C}
\end{displaymath}
is a fiber diagram with $A, B, C \in \Amod$ with trivial $\pi_0$ then it is in fact a cofiber diagram.  Thus, any object will preserve the diagram.  One then checks the remaining case and sees that Tor-amplitude of $0$ translates to preserving these fiber diagrams.

{\it 2)}  This is a consequence of the equivalence $M\otimes_{A} N \cong M \otimes_{\pi_0(A)} \pi_0(A) \otimes_A N$ for any $N \in \QC^0(A)$. 

{\it 3)} Since the statement is true for $\pi_0(A)$, this follows from above. 
\end{proof}

\begin{definition}
\label{nspc}
 $\sF \in \QC(X)$ is strict $n$-pseudo-coherent if there exists a filtered object $F$ of $\QC(X)$ with  $\sF \cong \varinjlim F$ and $F$ satisfying 
 \begin{enumerate}
 \item $F(i) = 0$ for $i << 0$,
 \item $F(i) \in \Perf(X)$ for $i \leq n$,
 \item $\cofib(F(i-1) \rightarrow F(i))$ has Tor-amplitude $-i$ for $i \leq n$
 \item $\cofib(F(i-1) \rightarrow F(i))$ has Tor-amplitude in $(-\infty, -i]$ for $i > n$.
 \end{enumerate}
\end{definition}

\begin{definition}
\label{spc}
 $\sF \in \QC(X)$ is strict pseudo-coherent if there exists a filtered object $F$ of $\Perf(X)$ with $\sF \cong \varinjlim F$ and $F$ satisfying  
 \begin{enumerate}
 \item $F(i) = 0$ for $i << 0$,
 \item $\cofib(F(i-1) \rightarrow F(i))$ has Tor-amplitude $-i$ for all $i$.
 \end{enumerate}
\end{definition}

In the case that $X$ is affine, our pseudo-coherence is equivalent to cellular objects in \cite{EKMM}.   We chose this definition (among many equivalent ones) since it allows one to naturally work with resolutions by perfect objects.  This allows much of the machinery built in \cite{TT} to be ported over with little work.  In particular, we need an analog of $\sigma^p$ (the brutal truncation).  This is clearly given by $F(p)$ for our strict pseudo-coherent objects.  Much like the classical case, the brutal truncation depends on the resolution and is not an invariant of the pseudo-coherent object.

\begin{definition}
For $\sF \in \QC(X)$, $\sF$ is pseudo-coherent ($n$-pseudo-coherent) if there exists an open affine cover $U$ of $X$ with $\sF|_U$ strict pseudo-coherent ($n$-pseudo-coherent).  We denote by $\pscoh(X)$ the full subcategory of $\QC(X)$ consisting of pseudo-coherent objects.
\end{definition}
\begin{remark}
By Lemma \ref{portingTT}, $\pscoh(X)$ is a stable and idempotently closed. 
\end{remark}
In the case that $X$ is ``underived'', i.e., $X \cong t_0(X)$, it is clear that our definition and that given in \cite{SGA6} are equivalent.  The definition we have given incorporates a resolution into the definition.  It is similar to the idea of a $p$-resolution in triangulated and dg-categories.

\begin{remark}
This definition is essentially the same as given in \cite{lurie_thesis}, but not the same as the definition given in \cite{tolly}.  Proposition \ref{aligned} gives the cases in which these coincide. 
\end{remark}

We now outline key properties of $\pscoh(X)$.  The interested reader is referred to the appendix for a more complete description of the properties of pseudo-coherent objects.  The proofs of the following two propositions can be found in the appendix as well.  Recall from \S \ref{notation}, $\iota_X: t_0(X) \to X$ is the natural counit morphism.


\begin{lemma}
\label{pullback}
If $\sF$ is pseudo-coherent on $X$, then $\iota_X^\ast(\sF)$ is pseudo-coherent on $t_0(X)$.
\end{lemma}

\begin{lemma}
\label{perfect_finite_tor}
Let $X$ be a geometric $D^-$ stack.
\begin{enumerate}
\item $\Perf(X) \subset \pscoh(X)$ and is defined by the property of locally finite Tor-amplitude.
\item $\sF \in \pscoh$ then $\tau^k(\sP)$ is $-k$-pseudo-coherent for all $k$. 
\end{enumerate}
\end{lemma}

\begin{example}
  Let $X$ be a geometric $D^-$-stack and $\sF$ perfect with Tor-amplitude $[0,0]$.  Then $\sF$ is a locally free sheaf on $X$.  The result for $\sF$ not necessarily strict follows from the fact that the operator $\tau^k$ commutes with flat pullback.
\end{example}



\begin{proposition}
\label{aligned}
Let $X$ be a quasi-compact stack with an ample sequence.  Then $\sF \in \pscoh(X)$ if and only if it is strict pseudo-coherent.
\end{proposition}



\begin{proof} 
Clearly strict pseudo-coherence implies pseudo-coherence. For the other implication, let $\sF \in \pscoh(X)$; we will prove strict pseudo-coherence by induction.  To start the induction, since $X$ is quasi-compact we can assume that $\sF$ is cohomologically bounded above.  Let $N$ be the highest non-trivial cohomology (i.e., $N$ is the least integer satisfying $\tau^k \tau_k(\sF) \cong 0$ for all $k > N$).  By definition of an ample sequence, we have a morphism $\oplus \sL_{i_\alpha}[-N] \xrightarrow{a_N} \sH^n(\sF)$, surjective on the $Nth$ cohomology.  Since $\iota_X^\ast \sF \in \pscoh(t_0(X))$, $\sH^n(\sF)$ is finite type and the $\iota_X^\ast \sL_i$ form a ample sequence on $t_0(X)$, thus we can assume this sum is over a finite index.  Let $F(-N) = \oplus  \sL_{i_\alpha}[-N]$ and $F(k) = 0$ for $k < -N$.

Assume that we have constructed $F(n)$, we will construct $F(n+1)$.  From $F(n) \xrightarrow{a_n} \sF$, we have the distinguished triangle
\begin{displaymath}
  \xymatrix{\fib(a_n) \ar[r] & F(n) \ar[r] & \sF }
\end{displaymath}
By Lemma \ref{conversion} $a_n$ is a $-n$-quasi-isomorphism, thus $\fib (a_n)$ is cohomologically bounded above by $-n$.  Using the same procedure as above, we have a morphism from a finite sum $\oplus_\beta \sL_{i_\beta}[n] \xrightarrow{b} \fib(a_n)$.  This yields a diagram

\begin{displaymath}
  \xymatrix{\sF\ar@{-->}[rd] \ar[dd] & &\ar[ll]^{a_n} F(n) \\ & \fib(a_n) \ar[ru] \ar[ld] & \\ \cofib(b)\ar@{-->}[rr]& & \ar[lu]^{b} \ar[uu] \oplus_\beta \sL_{i_{\beta}}}
\end{displaymath}
where the triangles involving dotted arrows are distinguished, and the dotted morphism involves a shift.  The octahedral axiom shows the existence of the diagram, which we use to define $F(n+1)$:
\begin{displaymath}
 \xymatrix{\sF \ar@{-->}[dd] & & F(n) \ar[ll] \ar[ld] \\ & F(n+1) \ar[lu] \ar@{-->}[rd] & \\ \cofib(b)\ar[rr] \ar[ru]& & \ar[uu] \oplus_\beta \sL_{i_{\beta}}[n]}
\end{displaymath}
$F(n+1)$ is perfect and $\cofib(F(n) \to F(n+1)) \cong \oplus_\beta \sL_{i_\beta}[n+1]$, which has Tor-amplitude of $-n -1$.  Further, the upper left diagonal morphism is the required structure morphism $F(n+1) \to \sF$, and the induction is complete.
\end{proof}

\begin{corollary}
\label{tiein}
  $\sF \in \pscoh(X)$ if and only if it is $n$-pseudo-coherent for all $n$.
\end{corollary}
\begin{proof}
 Clearly $\sF$ being pseudo-coherent implies it is $n$-pseudo-coherent.  For the other direction, note that $n$-pseudo-coherence is stable under pullback (see Lemma \ref{local_prop} for more detail).  Thus on any open affine cover $f_\alpha: U_\alpha \to X$, $f^\ast_\alpha \sF$ is $n$-pseudo-coherent for all $n$.  $U_\alpha$ satisfies the condition of Proposition \ref{aligned}.  It's proof didn't use the full strength of pseudo-coherence, only $n$-pseudo-coherent for each $n$.  Thus, $f^\ast_\alpha \sF$ is strict pseudo-coherent, and $\sF$ is pseudo-coherent. 
\end{proof}

\begin{proposition}
   Let $X$ be a quasi-compact perfect stack with ample bundle.  Then $\sF \in \pscoh(X)$ if and only if for any $k$, $\tau^{k}\sF$ is compact in $\QC^{[k, \infty)}(X)$.
\end{proposition}
 \begin{proof}
This follows easily from Proposition \ref{aligned}.  Mainly, we know the $\sF$ is strict pseudo-coherent.  By Corollary \ref{spcspcalt}, we know that $\tau^k\sF \cong \tau^k \sP$ for some $\sP \in \Perf(X)$.  The result follows from the compactness of $\sP$ (here is where the perfect stack assumption is used).  For the other direction, the quasi-compactness implies $\sF$ is cohomologically bounded above.  Let $N$ be the highest non-trivial cohomology sheaf.  Using that $\sF \cong \varinjlim \sP_\alpha$ and the compactness of $\tau^k(\sF)$ (in $\QC^{[k, \infty)}$), we know $\tau^k(\sF)$ is a retract of $\tau^k(\sP_\alpha)$.  By Proposition \ref{perfect_finite_tor}, the latter object is strict $-k$-pseudo-coherent.  Since strict $-k$-pseudo-coherence is stable under summands, $\tau^k(\sF)$ is strict $-k$-pseudo-coherent.  

Let $F_k$ be an associated choice of filtration for $\tau^k(\sF)$, and $\{i_\alpha: U_\alpha \to X \}$ an affine covering.  It is enough to verify that $i_\alpha^\ast F_k$ can be lifted to a filtration of $i_\alpha^\ast \sF$.    For $j \leq -k$, it is clear that $i^\ast_\alpha F_k(j)$ is perfect and has Tor-amplitude in $[-j, \infty)$.  Since we are working in the affine case, $\mr{Hom}(F(j), \tau_{k-1}(\sF)[1]) \cong 0$ (in $Ho(\QC(X))$).  The triangulated axioms show the morphism $F_k(j) \to \tau^k(\sF)$ factors as $F_k(j) \to \sF \to \tau^k(\sF)$ for $j \leq -k$, showing the desired lift.  The proposition now follows from Corollary \ref{tiein}.
\end{proof}

\subsection{Relative Tor amplitude and the bounded derived category}

Given $X$, we are interested in defining a derived geometric stack that classifies objects in $\bderive{X}$.  To do so, we need to define the correct notion of families of objects.  This section will set the ground work for this definition.

\begin{lemma}
\label{discrete}
If $K \in \QC^0(A)$ then $K$ is naturally a $\pi_0(A)$ module.  Geometrically, if $K \in \QC^0(A)$, then $K \cong\iota_{A\ast}K^\prime$.
\end{lemma}
\noindent This lemma clearly glues to give a similar statement for any $D^-$-stack.

\begin{definition}
Let $f: X \rightarrow Y$ be a morphism of $D^-$-stacks and $M \in \QC(X)$.  $M$ has Tor-amplitude $[a,b]$ over $f$ if for all $K \in QC^0(Y)$,  $M\otimes f^\ast K \in \QC^{[a,b]}(X)$.
\end{definition}

One may be tempted to declare an object to have finite relative Tor if it does under pullback to the underived case.  This would be incorrect.  This would lead to $B \in \kalg$ being flat over $A \in \kalg$ if and only if $\pi_0(B)$ is over $\pi_0(A)$.  This definition then would not agree with the framework of \cite{HAGII}.

\begin{example}
If $f: \spec A \rightarrow \spec A$ is the identity morphism, then this is the Tor-amplitude discussed earlier.
\end{example}


For the next proposition we need the following diagram.  If $f$ is a morphism of $D^-$-stacks, we have a Cartesian square
\begin{displaymath}
\xymatrix{ X \times_Y t_0(Y) \ar[d]^{g} \ar[r]^{\quad \epsilon} & X \ar[d]^f \\ t_0(Y) \ar[r]^\iota_Y & Y }
\end{displaymath}
The following proposition shows that relative Tor-amplitude has similar properties to the absolute case (Lemma \ref{absolute_tor}):
\begin{proposition}\mbox{}
\label{relative_tor_prop}
\begin{enumerate}
\item If $\s{F}$ has relative Tor-amplitude of $[a,a]$, then $\sF$ is flat over $Y$.
\item If $f$ is perfect, $\s{F}$ has relative Tor-amplitude of $[a, b]$ over $f$ if and only if the same is true of $\epsilon^\ast \s{F}$ over $g$.
\item If $f$ is quasi-compact and $Y$ is quasi-compact, $\s{F}$ perfect implies it has finite relative Tor-amplitude.
\item If $\sF \rightarrow \sG \rightarrow \s{H}$ is a cofiber sequence, then 
  \begin{enumerate}
    \item if $\sG$ and $\sH$ have relative Tor-amplitude of $[a,b]$, then $\sF$ has relative Tor-amplitude of $[a, b+1]$
    \item if $\sF$ and $\sH$ have relative Tor-amplitude of $[a,b]$, then so does $\sG$.
  \end{enumerate}
\end{enumerate}
\end{proposition}
\begin{proof} \mbox{}

1) Clear.

2) Let $K \in \QC(Y)$ and $K^\prime \in \QC(t_0(Y))$ such that $\iota_{Y\ast}(K^\prime) \cong K$.  We have a string of isomorphisms
\begin{align*}
M \otimes f^\ast \iota_{Y\ast} K^\prime & \cong M \otimes \epsilon_\ast g^\ast K^\prime \\
& \cong \epsilon_\ast (\epsilon^\ast M \otimes g^\ast K^\prime)
\end{align*}
Where the first equivalence is through base change (since $f$ is a perfect morphism) and the second equivalence is the projection formula (using that $\epsilon$ is a closed immersion).  Lastly, since $\epsilon_\ast$ is exact, one gets the result.

3)By Proposition \ref{perfect_finite_tor}, $\sF$ has locally finite Tor-amplitude.  The morphism, being quasi-compact, implies that over each affine of $Y$ we can find an finite relative Tor-amplitude.  Since $Y$ is quasi-compact, we can find a relative Tor-amplitude for all of $Y$.

4) These are clear from the long exact sequences on Tor groups.
\end{proof}

\begin{definition}
  Given a morphism $f: X \rightarrow Y$ of geometric $D^-$ stacks. Let $\bderivefl{X}$ denote the full stable sub-$\infty$-category consisting of pseudo-coherent objects with finite relative Tor dimension over $f$.
\end{definition}
If $X \in d\mr{St}_{Y}$ and $f$ is the structure morphism, then we drop the subscripts and refer to this as the bounded derived category of X, $\bderive{X}$.

\begin{remark}
This is a stable subcategory from Lemma \ref{relative_tor_prop}. 
\end{remark}

\subsubsection*{The case of a derived scheme flat over the base.}
To justify the choice of terminology, we suppose suppose that $X$ is a quasi-compact proper (or perfect) derived scheme over $R$ such that the structure morphism is flat.  The flatness of $f$ ensures that $t_0(X) \cong t_0(\spec R) \times_{\spec R} X$.  The properness (or perfection) of $X$ guarantees a proper base change.
Applying Lemma \ref{relative_tor_prop} shows that $\bderive{X}$ is the preimage of $\bderive{t_0(X)}$ under the natural pullback morphism.  

This latter category is exactly the well-known bounded derived category if $X$ is Noetherian.  If $X$ is not Noetherian, then this falls into the cases studied in \cite{SGA6}.  One doesn't necessarily need the structure morphism to be flat; the presentation is considerably simpler when it is. 


\begin{example}
 If $X$ is quasi-compact with $\s{O}_X$ finite Tor-amplitude over $f$, then $\bderivefl{X}$ contains $Perf(X)$.  If $f$ is the identity, then this inclusion is an equivalence.
\end{example}







\subsection{Corepresentability}

\begin{lemma}
\label{strong_dual}
 Let $f: X \rightarrow Y$ satisfy the following
\begin{enumerate}
\item $f$ is quasi-compact and quasi-proper (i.e.,  takes pseudo-coherent to pseudo-coherent) with $\sO_X$ bounded relative Tor-amplitude over $f$.
\item locally $f$ is finite cohomological dimension.
\end{enumerate}
Then $\sF \in \bderivefl{X}$ implies $f_\ast \sF$ is perfect.
\end{lemma}
\begin{proof}
It suffices to work locally, thus we can assume $Y \cong \spec A$ and $f$ has cohomological dimension $h$.  If $\sF \in \bderivefl{X}$, by Lemma \ref{perfect_finite_tor} we just need to show that it has finite Tor-amplitude.  The projection formula ensures that $f_\ast(\s{P}) \otimes_{A} N \cong f_\ast(\s{P} \otimes_{\s{O}_X} f^\ast N)$.  If $N \in \QC^0(Y)$, $\s{P} \otimes_{\s{O}_X} f^\ast N \in \QC^{[a.b]}(X)$, thus $f_\ast(\s{P}) \otimes_{\s{O}_Y} N \in \QC^{[a, b+h]}$.  Lemma \ref{perfect_finite_tor} then shows that $f_\ast \sP$ is perfect.

\end{proof}

\begin{corollary}
If $\sF \in \bderivefl{X}$ then $f_\ast \sF$ is perfect for the following types of maps:
\begin{enumerate}
\item $X$ and $Y$ are quasi-compact and quasi-separated, with $f$ proper and perfect (in the sense of \cite{SGA6}) \cite{LN}.
\item $f$ quasi-compact, quasi-separated with $f$ categorically proper (i.e., takes perfect objects to perfect objects).
\end{enumerate}
\end{corollary}


\begin{example}
Any projective variety flat over a Noetherian base satisfies the conditions of Lemma \ref{strong_dual}.
\end{example}

The next proposition is the key result needed to ensure our moduli will be even remotely well behaved.  We preface it with a discussion about mapping spaces and mapping sheaves in symmetric monodial stable $\infty$-categories.

Recall that given two objects $P, Q \in \QC(X)$, $\Map(P, Q)$ is only defined up to homotopy type.  A key component of a simplicial model category is the ability to obtain functorial mapping spaces.  When working with $\infty$-categories functorial mapping spaces (that are functorial in both arguments) becomes quite technical.  Fortunately, if we restrict to varying only one argument, there are easy descriptions of the mapping spaces: if $Map^L(P, -)$ is a choice of functor associated to the undercategory $\QC(X)_{P/}$ and $\Map^R(-, Q)$ is a choice of functor associated to the overcategory $\QC(X)_{/Q}$, then one has
\begin{displaymath}
  \Map^R(P, Q) \cong \Map(P, Q) \cong \Map^L(P, Q).
\end{displaymath}
One can carry this a bit further: $\QC(X)$ is a symmetric monodial presentable $\infty$-category for which the monodial structure preserves colimits in each variable.  The adjoint functor theorem applies to give an enriched mapping space when one of the arguments is fixed.  We denote a choice of these enriched mapping spaces as $\sdHom(P, -)$ and $\sdHom(-, Q)$.

\begin{proposition}
\label{corepresentable}
Suppose $f$ satisfies the conditions of Lemma \ref{strong_dual} with $X$ perfect and $Y$ quasi-compact.  For any $M, N \in D^b_{fl, f}(X)$, the functor $\QC(Y) \rightarrow \QC(Y)$ given by $f_\ast(\sdHom(M, N\otimes f^\ast - ))$ is corepresentable.
\end{proposition}
\begin{proof}
Since $X$ is perfect, we have $M \cong \varinjlim P_\alpha$ with $P_\alpha \in Perf(X)$.  By Lemma \ref{relative_tor_prop}, $\Perf(X) \subset D^b_{fl, f}(X)$.   It is an easy calculation that $Perf(X) \otimes \bderivefl{X} \subset \bderivefl{X}$. Thus,
\begin{align*}
f_\ast \sdHom_X(M, N\otimes f^\ast - )& \cong f_\ast \sdHom_X(\varinjlim P_\alpha, N\otimes f^\ast -)  \\ 
& \cong f_\ast (\varprojlim \sdHom_X(P_\alpha, N\otimes f^\ast -))  \\
& \cong  \varprojlim f_\ast \sdHom_X (P_\alpha, N \otimes f^\ast - )  \\
&  \cong \varprojlim (f_\ast ((P_\alpha^\vee \otimes N) \otimes f^\ast -))  \\
 & \cong \varprojlim (f_\ast (P_\alpha^\vee \otimes N) \otimes -) \\
\end{align*}
By Lemma \ref{strong_dual}, $f_\ast (P_\alpha^\vee \otimes N)$ is dualizable.  Thus
\begin{align*}
\varprojlim (f_\ast (P_\alpha^\vee \otimes N) \otimes -) & \cong \varprojlim (\sdHom_Y(f_\ast(P_\alpha^\vee \otimes N)^\vee, -))  \\
 & \cong \sdHom_Y(\varinjlim (f_\ast(P_\alpha^\vee \otimes N))^\vee, -)
\end{align*}
Therefore, letting $Q := \varinjlim ((f_\ast(P_\alpha^\vee \otimes N))^\vee)$, then $Q \in QC(Y)$ and satisfies $f_\ast \sdHom(M, N\otimes f^\ast - ) \cong \sdHom(Q, -)$.
\end{proof}				   
\begin{remark}
The assumptions in Proposition \ref{corepresentable} are stronger than need be;  it is true for quasi-compact and quasi-separated derived schemes satisfying Lemma \ref{strong_dual}.  This follows from the results in \cite{LN} that show pseudo-coherent objects can be approximated by perfect objects on these schemes.  Since our moduli in this paper will be for projective morphisms, this extension will be relegated to future work.
\end{remark}

\subsection{Properties of the corepresenting object}
For this section, $f: X \rightarrow \spec A$ satisfies the condition of Proposition \ref{corepresentable} and $X$ has an ample sequence.
\begin{lemma}
\label{itpsuedo}
 Suppose $M, N$ have Tor-amplitude of $[a,b]$ and $[a^\prime, b^\prime]$ over $f$ (respectively), then the corepresenting object $Q$ is pseudo-coherent and is cohomologically bounded above by $b - a^\prime$.
\end{lemma}

\begin{proof}
Clearly, $\spec A$ is perfect and quasi-compact.  Thus, by Proposition \ref{aligned} it suffices to check that $R\mr{Hom}_A(Q, -)$ commutes with filtered colimits with bounded total cohomology.  This is a consequence of the chain of equivalences 
\begin{align*}
  R\mr{Hom}_A(Q, \varinjlim_\alpha G_\alpha) & \cong R\mr{Hom}_{X}(M, N\otimes_{\s{O}_X} f^\ast(\varinjlim_\alpha G_\alpha)) \\
& \cong R\mr{Hom}_{X}(M, N\otimes_{\sO_X} \varinjlim_\alpha f^\ast(G_\alpha)) \\
& \cong R\mr{Hom}_{X}(M, \varinjlim_\alpha N\otimes_{\sO_X} f^\ast(G_\alpha))
\end{align*}
By definition of $N$ there exists a $c,d$ such that for each $\alpha$, $N\otimes f^\ast(G_\alpha) \in \QC^{[c,d]}(X)$.  From the proof of Proposition \ref{aligned}, $M \in \pscoh(X)$ implies
\begin{displaymath}
  R\mr{Hom}_{X}(M, \varinjlim_\alpha N\otimes_{\sO_X} f^\ast G_\alpha) \cong \varinjlim_\alpha R\mr{Hom}_{X}(M, N\otimes_{\sO_X} f^\ast G_\alpha) \cong \varinjlim_\alpha R\mr{Hom}_A(Q, G_\alpha).
\end{displaymath}
Thus, $Q \in \pscoh(A)$.


To show that the cohomology of $Q$ is bounded above by $b -a^\prime$, first note that for any perfect $P \in \QC(X)$, Proposition \ref{absolute_tor} implies if $P$ has Tor-amplitude of $[a,b]$ then $P \in \QC^{(-\infty, b]}$.  This in turn implies the same if $\sP \in D^b_{fl, f}$ (with the same Tor-amplitude).

Let $K \in \QC^{[k, k]}(A)$ with $k > b-a^\prime$.
 Then
 \begin{align*}
   R\mr{Hom}_A(Q, K)  \cong R\mr{Hom}_{X}(M, N\otimes_{\sO_X}f^\ast K)  \cong 0
 \end{align*}
since $N\otimes_{X}f^\ast K \in QC^{[a^\prime + k, b^\prime + k]}$ with $a^\prime + k > b$ and from above $M \in \QC^{(-\infty, b]}(X)$. This property along with $Q$ being pseudo-coherent is enough to show the result.
\end{proof}

The importance of the corepresenting object is the natural stack associated to it.  
\begin{proposition}
\label{geometric_stack_rep}
Let $Q \in \pscoh(A)$ be cohomologically bounded above with upper bound $n$.  Then $\Map^L(Q, -): \Aalg \rightarrow \s{S}$ is a geometric $n$-stack of almost finite type.  
\end{proposition}
\begin{proof}
If $Q \in \QC^{(-\infty, 0]}(A)$, then this functor is represented by the free symmetric algebra on $Q$, and thus an affine derived scheme.  By Corollary \ref{spcspcalt} there exists a perfect $P$ $n$-quasi-isomorphic to $Q$ for all $n \in \mZ$.  Thus if $m > n$, 
\begin{displaymath}
\Map^L(P, -)|_{\QC^{[m, \infty)}(A)} \cong \Map^L(Q, -)|_{\QC^{[m, \infty)}(A)}
\end{displaymath}
By \cite{HAGII, dg_moduli} the former functor represents a geometric $n$-stack of finite type.  Thus, $\Map^L(Q, -)$ is a geometric $n$-stack of almost finite type. 
\end{proof}

\section{The moduli stack of pseudo-coherent sheaves.}

We now let $k  \in \mZ\textrm{-Alg}_{\Delta}$ and $X$ be a geometric stack.   From Appendix \ref{bicart}, there exists a  biCartesian (Cartesian and coCartesian) fibration $\QCF$ over $\kAlg$ whose fiber over $A \in \kAlg$ is the symmetric monodial $\infty$-category $\QC(X \times \spec A)$ and coCartesian (Cartesian) morphisms are pullback (pushforward, respectively).  
If $\sF \in \QCF$, then $\sF$ can be considered in $\QC(X \times \spec \pi(\sF))$ in a natural way.  It is well known that $\sM_{\QCF}$, as a left fibration is a $D^-$-stack.

Define $\widetilde{\bderive{X}}$ to be the full sub-$\infty$-category of $\QCF$ consisting of pseudo-coherent objects with finite Tor-amplitude over the second projection.  It is not immediate that $\widetilde{\bderive{X}}$ is a coCartesian fibration.  Accepting this for the moment, $F_{\widetilde{\bderive{X}}}$ is equivalent to the functor  
\begin{align*}
   A \rightarrow \bderiveflpi{X \times \spec A}
\end{align*}
which assigns to morphisms the natural pullback functor $f^\ast$ (i.e., $\otimes_A B$).

\begin{definition}
Define $\sM_{\bderive{X}}$ as the classifying left fibration (over $\kAlg$) associated to $\widetilde{\bderive{X}}$. 
\end{definition}

\begin{lemma}
\label{local_prop}
Given $f: A \rightarrow B$, and $F \in \bderiveflpi{X \times \spec A}$ then $\s{F} \otimes_A B \in \bderiveflpi{X \times \spec B}$.  If $f$ is a faithfully flat and finitely presented, then $\s{F}\otimes_A B \in \bderiveflpi{X \times \spec B}$ implies $F \in \bderiveflpi{X \times \spec A}$.
\end{lemma}
\begin{proof}

First, pseudo-coherence is stable by base change: perfect objects are stable under pullback.  Let $\sF$ be a pseudo-coherent object on $X \times \spec A$, with $U_i$ a cover such that $F\big|_{U_i}$ is strict pseudo-coherent.  Let $F_i$ be a choice of corresponding filtered object in $\QC(U_i)$.  If we choose the cover $F\big|_{U_i \times \spec B}$ of $X\times \spec B$, then we have the following filtered object on $U_i \times_{\spec A} \spec B$: $F^\prime_i(n) = F_i(n)\otimes_A B$.   Observe that if $P$ is perfect and Tor-amplitude of $[a, b]$, then $f^\ast P$ will be the same (this is essentially the statement that the pullback of a vector bundle is a vector bundle).  Since pullback commutes with colimits, we just need to ensure that $F_i^\prime$ satisfies the necessary properties.  This is obvious from the previous statement. 

Thus, to show that $\sF\otimes_A B \in \bderiveflpi{X \times \spec B}$, we just need to ensure it is bounded Tor-amplitude.  By Lemma \ref{relative_tor_prop}, it suffices to assume $A$ and $B$ are discrete.  In this case, it is clear from the definition: $\s{F} \otimes_A B \otimes_B M \cong \s{F} \otimes_A M$, which is bounded.  In the case that $B$ is faithfully flat and finitely presented,  suppose $\s{F}$ is not finite relative Tor-amplitude.  Then for any $m$, there exists an $i < m$ and $N \in \QC^{[0, 0]}$ with $H^i(F\otimes N)\neq 0$.  Since $B$ is assumed to be fppf, $H^i(F\otimes_A B \otimes_B N) \cong B \otimes H^i(F \otimes_A N)$, which is not zero by definition of $B$.   The local nature of pseudo-coherence ensures that if $\sF \otimes_A B$ is pseudo-coherent then $\sF$ is as well.

\end{proof}

That $\widetilde{\bderive{X}}$ is a coCartesian fibration follows from the above lemma.  Since it is known that $\sM_{\QCF}$ is a $D^-$-stack (in the $fppf$ topology), paired with the proposition below the lemma also shows that $\sMtor$ is a $D^-$-sub-stack of $\sM_{\QCF}$.  

\subsubsection*{Substacks}

We will repeatedly be forming new functors via classifying fibrations obtained by full subcategories of $\QCF$.  Here we will give a criterion for these fibrations to be stacks.  This is used implicitly many times in \cite{HAGII}.  We have written it down for completeness.  For this section, let $S \in (\sset)_{/\kAlg}$ be a coCartesian fibration.  As always, $\sM_{S}$ will be the associated classifying left fibration and $F_{S}$ the associated functor.  

\begin{proposition}
\label{locally_determined}
Suppose $K \in (\sset)_{/\kAlg}$ such that $K \subset S$ is a full and faithfull sub-coCartesian fibration closed under equivalences. Then $K$ is a stack if 
\begin{enumerate}
\item $\sM_{S}$ is a stack and
\item suppose $M \in S$ lies over $A$ and $f: M \rightarrow N$ is a coCartesian arrow lifting a $fppf$ morphism in $\kAlg$.  Then if $N \in K$ then $M \in K$.
\end{enumerate}
\end{proposition}
The utility of this proposition is that one doesn't need to verify descent for hypercovers, but only on covers.

\begin{proof}
It is more convenient to work with the functor viewpoint, rather than the fibration.  To align notation with previous work, for this proof we will let $\spec A$ for $A \in \kAlg$ denote the representable functor associated to $A$.  $F_{\sM_{K}}$ and $F_{\sM_S}$ are functors from $\kAlg$ to $\s{S}$ (the $\infty$-category of Kan complexes). After this conversion, the second condition is equivalent to $M \in F_{\sM_{K}}(A)$ and $f$ a $fppf$ morphism, $M \in F_{\sM_{K}}(A)$ if $F_{\sM_S}(f)(M) \in F_{\sM_{K}}(B)$.  

Let $\s{H}$ be a hypercover of $A \in \kAlg$ (in the $fppf$-topology) with $\s{H}_n = \coprod_{\alpha \in I_n} \spec B_{\alpha,n}$.  Then one has the commutative diagram
\begin{displaymath}
  \xymatrix{ F_{\sM_{K}}(A) \ar[r] \ar[d] & \varprojlim F_{\sM_{K}}(\s{H}_n) \ar[d] \\ F_{\sM_{S}}(A) \ar[r] & \varprojlim F_{\sM_{K}}(\s{H}_n)}
\end{displaymath}
Here $F_{\sM_{K}}(\s{H}_n)$ denotes $\prod_{I_n} F_{\sM_{K}}(B_{\alpha,n})$.  Our assumptions that $K$ is a full subcategory of $S$ ensures that if any face/edge of a $m$-simplex is contained in $F_{\sM_{K}}$, then so is the whole simplex.  This property is preserved under the above limit.  This can be seen by using the explicit limit formula given in \cite[\S 1.2]{topoi} and thus is true for both vertical arrows.

We now show $F_{\sM_{K}}(A) \cong \varprojlim F_{\sM_{K}}(\s{H}_n)$.  First,   
\begin{displaymath}
\pi_0(F_{\sM_{K}}) \rightarrow \pi_0(\varprojlim F_{\sM_{K}}(\s{H}_n))
\end{displaymath}
is an isomorphism. Injectivity follows by applying $\pi_0$ to the above commutative diagram: the property discussed above implies the vertical morphisms are injections after applying $\pi_0$  and the bottom morphism is an isomorphism.  For surjectivity, note if $P: \spec A \rightarrow F_{\sM_{S}}(A)$ is a point of $F_{\sM_{S}}(A)$ with $P \in \varprojlim F_{\sM_{K}}(\s{H}_n)$ then the augmentation map $h_0: \s{H}_0 \to \spec A$ restricts to $fppf$-morphisms $\spec B_{\alpha,0} \to \spec A$.  Our assumptions then imply that $P \in F_{\sM_{K}}(A)$.

Isomorphisms on higher homotopy groups are easier.  The discussion above implies 
\begin{align}
F_{\sM_{K}}(A) \cong \pi_0(F_{\sM_{K}}(A)) \times^h_{\pi_0(F_{\sM_{S}}(A))} F_{\sM_{S}}(A) \\
\varprojlim F_{\sM_{K}}(\sH_n) \cong \pi_0(\varprojlim F_{\sM_{K}}(\sH_n)) \times^h_{\pi_0(\varprojlim F_{\sM_{S}}(\sH_n))} \varprojlim F_{\sM_{S}}(\sH_n).  
\end{align}
The result then follows from the fact $F_{\sM_{S}}$ is a stack. 
\end{proof}

\subsubsection*{The stack $\sMtor$}
The statement of Lemma \ref{local_prop} is the second condition of Proposition \ref{locally_determined}. Thus, 

\begin{proposition}
$\sM_{\bderive{X}}$ is a $D^-$ stack.
\end{proposition}
As mentioned in previous sections when $X$ is a scheme and $k$ is a field $\sMtor(\spec{k}) \cong N^w(\bderive{X})$ (e.g., $\pi_0(\sMtor(\spec{k})) \cong \mr{Iso}(\bderive{X})$.  Thus, our fibration classifies the correct objects.  Specializing to $X$ smooth, we have the following comparison, showing that our fibration aligns with previously calculated examples.

\begin{lemma}
If $X$ is smooth and flat over $\spec k$, with $k$ discrete, then $\sM_{tor,\bderive{X}} \cong \sM_{perf}$, as constructed in \cite{Pranav} and \cite{dg_moduli}.
\end{lemma}
\begin{proof}
  This is clear once one observes that $D^b_{fl, \pi}{X \times \spec A} \cong Perf(X \times \spec A)$ by an easy extension of  \cite[Proposition III.3.6]{SGA6}.
\end{proof}

\begin{remark}
An example of an object that would not be in this category is if $X$ is singular, and $\spec{B}$ is a local Zariski open set containing part of the singularity.  Let $\sF$ be diagonal of the singular locus.
\end{remark}


\section{Geometricity of $\sMtor$}

By abuse of notation, given a morphism $Q: \Aalg \rightarrow \sMtor$ we will refer to the pseudo-coherent sheaf on $X \times \spec A$ as $Q$ as well
.  Further, we will denote $X \times \spec A$ by $X[A]$.  

For a general $X$, there is no hope that $\sMtor$ is a reasonable derived stack (e.g., finite type).  We will show that using the following definition for a projective derived scheme, $\sMtor$ is indeed a reasonable stack.

\begin{definition}
  A derived projective scheme $X$ over $Y$ with structure morphism $X \stackrel{f}{\rightarrow} Y$ is a derived scheme with a locally free rank 1  quasi-coherent sheaf $L$ (i.e., for any map $f: \spec A \rightarrow X$, $f^\ast \sL$ is a projective rank 1 module) satisfying if $\sF \in \bderivefl{X} \cap \QC^{(-\infty, 0]}(X)$ then there exists a $m \in \mZ$ with $f_\ast (L^{\otimes n}\otimes_X \sF) \in \QC^{(-\infty, 0]}(Y)$ for all $n > m$.
\end{definition}

\begin{remark}
  For the case we are working with, $Y = \spec k$ and $X$ flat over $k$, then $X$ projective over $\spec k$ implies $X\times \spec A$ is projective over $\spec A$.  This follows from the well known property in the discrete case and the standard simplicial Tor spectral sequence.
\end{remark}


\begin{theorem}
\label{itworks}
Let $X \rightarrow \spec k$ be a projective derived scheme over a simplicial commutative algebra $k$ with structure morphism satisfying the conditions of Lemma \ref{strong_dual} and flat over $k$, then $\sMtor$ is a locally geometric stack $D^-$ $k$-stack of almost finite type.
\end{theorem}


The proof of this theorem will occupy the next sections.  We will define a Zariski open covering of $\sMtor$ and show geometricity of each element of the covering.

\subsection{The substacks $\sMab$}
Given $a, b \in \mZ$ with $a <b$, one has full subcategories of $\bderiveflpi{X[A]}$ consisting of objects with relative Tor-amplitude contained in $[a,b]$.  We will denote this category by $\pscoh_{\pi}^{[a,b]}(X[A])$.

Given a morphism $A \rightarrow B$ in $\kAlg$, if $G \in \bderiveflab{X[A]}{a}{b}$ then $G\otimes_A B\otimes_B M \cong G\otimes_A M$ where the latter has a natural $B$-module structure. This is enough to show that $G \otimes B \in \bderiveflab{X[B]}{a}{b}$. Thus, we know that the full subcategory $\widetilde{\pscoh}^{[a,b]}(X)$ of $\widetilde{\bderive{X}}$ consisting of objects with relative Tor-amplitude between $[a, b]$ is a coCartesian fibration over $\kAlg$.  Taking classifying fibrations, we denote $\sMab := \sM_{\widetilde{\pscoh}^{[a,b]}(X)}$.  We have the following:

\begin{proposition}
\label{ab_is_zariski_open}
$\sMab$ is a Zariski open substack of $\sMtor$.
\end{proposition}
\begin{proof}
 It is easy to see from the proof of Lemma \ref{local_prop} that $\sMab$ satisfies the conditions of \ref{locally_determined}, and therefore is a substack of $\sMtor$. 

We need to show that $\sMab \rightarrow \sMtor$ is an immersion and finitely presented.  For the immersion, by definition we must show that $\sMab \rightarrow \sMab \times_{\sMtor} \sMab$ is an equivalence.  This is obvious as the morphism $\sMab \rightarrow \sMtor$ is a weak equivalence on all components with non-empty preimage.   To complete the proof, we will show that $\Aalg \times_{\sMtor} \sMab$ is a Zariski open subset of $\Aalg$ for any $\Aalg \stackrel{\sF}{\rightarrow} \sMtor$ with $A \in \kAlg$ (and thus show it is finitely presented).  

Since $\sMab$ the moduli of a full subcategory, the morphism $\sMab \rightarrow \sMtor$ is a left fibration.  

This is easily checked by analyzing the lifting conditions for left fibrations of simplicial sets.  As such, we have explicit models for the fiber product with any $\Aalg  \xrightarrow{\sF} \sMtor$: the fiber product is the fiber product as simplicial sets over $\kAlg$.  Further, if $\s{Z} := \sMab \times_{\sMtor} \Aalg$, then $\s{Z}$ fits into the diagram with both squares Cartesian 
\begin{displaymath}
  \xymatrix{ \s{Z} \ar[r] \ar[d]& \Aalg \ar[d] \\ \sMab \ar[r] \ar[d] & \sMtor \ar[d] \\ h(\sMab) \ar[r] & h(\sMtor)}
\end{displaymath}
where $h(\sMab)$ is the left fibration with $F_{h(\sMab)} = \pi_0(F_{\sMab})$ (well defined up to equivalence).  These two statements combined show that we can assume  $\s{Z} \subset \Aalg$ is the largest sub-simplicial set on the objects $\{f \in (\Aalg)_0 |\; f^\ast\sF \textrm{ has Tor-amplitude in [a,b]}\}$ and is closed under equivalence.  

Since $X[A]$ is flat over $A$, this implies $t_0(X) \cong \spec \pi_0(A) \otimes_{\spec A} X$.  Using the correspondence between Zariski open subschemes of $\spec A$ and $\spec \pi_0(A)$ combined with Proposition \ref{relative_tor_prop}(2) we can assume $X$ and $A$ are discrete.  However, with this description of $\s{Z}$, the result then easily follows from \cite[Lemma 2.1.4]{lieblich} (originally due to Grothendieck) since in this case, one can replace $\sF$ with a complex of $A$-flat $\sO_{X}$-modules.

\end{proof}

\begin{proposition}
$\sMab$ is a geometric $b-a + 1$ stack of almost finite type.
\end{proposition}
By almost finite type, we mean that $\sMab$ restricted to the full subcategory of $n$-truncated $k$-algebras is finite type for any $n \in \mathbb{N}$.  We begin by explicitly verifying the diagonal of $\sMab$ is $b-a$ representable.  We then proceed to construct an atlas via induction.   The proof is contained in the next two subsections.

\subsection{Representability of diagonal}

In this section we show the natural morphism $\sMab \rightarrow \sMab \times \sMab$ is $(b -a)$-representable.  Let $\mI_{P, Q}$ be a left fibration over $\Aalg$ fitting into the fiber square 
\begin{displaymath}
  \xymatrix{\mI_{P, Q} \ar[r]\ar[d] & \Aalg \ar[d]^{(P, Q)} \\ \sMab \ar[r] & \sMab \times \sMab}
\end{displaymath}
The $(b-a)$-representability of $\sMab$ amounts to showing that $\mI_{P,Q}$ is a $(b-a)$-geometric stack for all choices of $(P, Q)$.  Showing the geometricity of $\mI_{P, Q}$ will proceed much like the calculation that $\mr{GL}(n, k)$ is a variety: one first proves that $\mr{M}(n,k)$ (the set $n\times n$ matrices with coefficients in $k$) is an affine scheme and proceed to write $\mr{GL}(n, k)$ as the complement of a Zariski closed subscheme. 
For us, $\mr{M}(n,k)$ will take the form of $\mV_{P,Q}$ (defined below).  We will show the geometricity of this object and show that $\mI_{P, Q}$ is a Zariski open substack.  

Recall that the $\widetilde{\QC(X)}$, being Cartesian (and coCartesian) over $\kAlg$, is a fibrant-cofibrant element of the contravariant (covariant, respectively) simplicial model structure on the category $(\mr{Set}^+_\Delta)_{/\kAlg}$.  There are two different function complexes (cotensor) associated to a simplicial set $K$ corresponding to the two model structures on $\sset$: the standard and the Joyal.  We will refer to these as $\QCF^{K,\flat}$ and $\QCF^{K, \sharp}$, respectively.  Explicitly, $\QCF^{K, \flat}$ classifies the functor
\begin{displaymath}
  S \rightarrow \mr{Hom}_{\sset^+}(K^\flat \times S, \QCF) 
\end{displaymath}
while $\QCF^{\Delta_1, \sharp}$ classifies the functor
\begin{displaymath}
  S \rightarrow \mr{Hom}_{\sset^+}(K^\sharp \times S, \QCF).
\end{displaymath}
where $K^\flat$ and $K^\sharp$ are the the constant marked simplicial set over $\kalg$ with only degenerate edges marked (all edges marked, respectively).  Thus $\QCF^{\Delta_1, \sharp} \cong \QCF$ and $\QCF^{\Delta_1, \flat}$ is the $\infty$-category of morphisms (relative to $\kAlg$).  In both cases, a monomorphism $K \subset K^\prime$ induces fibrations $\QCF^{K^\prime, \flat} \rightarrow \QCF^{K, \flat}$ and $\QCF^{K^\prime, \sharp} \rightarrow \QCF^{K, \sharp}$. Further, there exists a natural inclusion $\QCF^{K, \sharp} \subset \QCF^{K, \flat}$ arising from the morphism $K^\flat \subset K^\sharp$.  We will many times abbreviate the mapping object $\QCF^{K, \flat}$ by $\QCF^{K}$.

Applying this to the natural cofibration $S^0 \rightarrow \Delta_1$.  We can factor
\begin{displaymath}
\QCF \rightarrow \QCF \times \QCF
\end{displaymath}
as
\begin{displaymath}
\QCF \rightarrow \QCF^{\Delta_1, \sharp} \rightarrow \QCF \times \QCF
\end{displaymath}
where the first morphism is an equivalence and the second is a fibration. Since the classifying functor is a right Quillen functor, it preserves fibrations.  Explicit calculation shows $\sM_{\widetilde{\QC(X)}^{K, \sharp}} \cong \sM_{\widetilde{\QC(X)}}^K$.  Thus the homotopy fiber product
\begin{displaymath}
 \sM_{\widetilde{\QC(X)})^{\Delta_1}, \sharp} \times^h_{\sM_{\QC(X)} \times \sM_{\QC(X)} } \Aalg
\end{displaymath}
is just a fiber product of simplicial sets over $\kAlg$.  Clearly this is equivalent to $\mI_{P, Q}$.  

Define $\mV_{P,Q}$ as the left fibration fitting into the fiber diagram 
\begin{align}
\label{morph_cat}
  \xymatrix{ \mV_{P, Q} \ar[r] \ar[d] & \Aalg \ar[d]^{(P, Q)} \\ \sM_{\widetilde{\QC(X)}^{\Delta_1, \flat}} \ar[r]^{s\times t\qquad} &  \sM_{\widetilde{\QC(X)}} \times \sM_{\widetilde{\QC(X)}}}.
\end{align}
  With this description, the functor corresponding to $\mV_{P, Q}$ (over $\Aalg$) is given by   
\begin{displaymath}
  \xymatrix{(\spec B \ar[r]^f & \spec A)\ar[r] &  \Map_{\QC(X[B])}(f^\ast P, f^\ast Q)}
\end{displaymath}
on the object level (we are not defining this functor, it already exists, we are just describing the homotopy type evaluated on objects).  The inclusion $\QCF^{\Delta_1, \sharp} \subset \QCF^{\Delta_1, \flat}$ then yields a natural inclusion $\mI_{P, Q} \rightarrow \mV_{P,Q}$.

\begin{proposition}
\label{homotopy_coherent}
  $F_{\mV_{P, Q}}$ is equivalent to the functor
  \begin{align*}
    \Aalg & \rightarrow  \s{S} \\
    B & \rightarrow  \Map_{X[A]}^L(P, Q\otimes_A B)
  \end{align*}
\end{proposition}
\begin{proof}
Since $\widetilde{\pscoh}^{[a,b]}(X)$ is a full sub-coCartesian fibration of $\widetilde{\QC(X)}$, we can assume we are working with the latter biCartesian fibration over $\kAlg$.  Let $\widehat G$ denote the left fibration defined by the fiber diagram 
\begin{displaymath}
  \xymatrix{ \widehat{G} \ar[r] \ar[d] & \sM_{\widetilde{\QC(X)}_{P/}} \ar[d]\\ \Aalg \ar[r]^{Q} & \sM_{\widetilde{\QC(X)}}}
\end{displaymath}
The fiber of $\widehat{G}$ over $B$ is $\Map^L_{\widetilde{\QC(X)}}(P, Q_A \otimes B)$.  There exists a natural morphism from $\mV_{P, Q}$ to $\widehat G$.  Mainly, the definition of the morphism $P: \Aalg \rightarrow \sM_{\QCF}$ as left fibrations over $\kAlg$ ensures the image of any morphism from the initial object ($A \xrightarrow{id} A$) is a coCartesian morphism $P \rightarrow P\otimes_A B$  in $\widetilde{\QC(X)}$.  The data of the our module $P$ is then a homotopy coherent choice of these coCartesian morphisms.  Since the fiber of $\mV_{P, Q}$ over $B$ is $\Map_{X[B]}(P\otimes_k B, Q\otimes_kB)$, using the data supplied by $P$, we can precompose to obtain a morphism $\mV_{P, Q} \to \widehat G$.  More precisely, one finds a lift in all possible compositions, giving a morphism that only depends on a contractible space of choices.  Since we are precomposing with coCartesian morphisms, the morphism $\mV_{P, Q} \to \widehat G$ will be a categorical equivalence.

The more tedious part of this proposition is the Cartesian adjunction.  In particular, to get the equivalence stated in the proposition we need to assemble adjunctions in a homotopy coherent fashion. Let $\Delta_1 \rightarrow \kAlg^{\Aalg}$ be the given as follows. It is clear that the limit over the natural morphism $\Aalg \rightarrow \kAlg$ is equivalent to $A$.  Thus, we have a limit diagram $\Aalg^{\vartriangleleft} \rightarrow \kAlg$ and a morphism $\Aalg \times \Delta_1 \rightarrow \kAlg$ that factors through this limit diagram. Evaluated on the two inclusions of $\Delta_0$, they are the constant functor $A$ and the natural morphism $\Aalg$ to $\kAlg$, respectively.  Note that the former will cease to be a left fibration over $\kAlg$.

By \cite[Proposition 3.1.2.1]{topoi} the functor category $\widetilde{\QC(X)}^{\Aalg}$ is Cartesian over $\kAlg^{\Aalg}$.  In particular, one has a lift diagram  
\begin{displaymath}
  \xymatrix{ \Delta_0 = \Lambda_1^1 \ar[d] \ar^Q[r] &\widetilde{\QC(X)}^{\Aalg} \ar[d]\\
\Delta_1 \ar[r] \ar^E@{-->}[ru] & \kAlg^{\Aalg} }
\end{displaymath}
where $E$ is a Cartesian edge.  By the same proposition we know that for any $B \in \Aalg$, $E(B) \in (\widetilde{\QC(X)})_1$ is Cartesian (over the natural map $\QCF \rightarrow \kAlg$).  Applying this to our particular morphism $\Delta_1 \rightarrow \kAlg^{\Aalg}$ the Cartesian edge requirement ensures that $E_{1}$ is a homotopy coherent composition of adjunctions: in the homotopy category the straightened version of this functor is $B \rightarrow Q \otimes B \in \QC(X[A])$.

Let $\widehat{F}$ fit into the fiber diagram
\begin{displaymath}
  \xymatrix{\widehat{F} \ar[r] \ar[d] & \sM_{\widetilde{\QC(X)}_{P/}} \ar[d] \\ \Aalg \ar[r]^{\sM(E|_{1})} & \sM_{\widetilde{\QC(X)}}}
\end{displaymath}
It is clear that $\widehat{F}$ corresponds to the functor 
\begin{displaymath}
B \rightarrow \Map^L_{X[A]}(P, Q \otimes B)\qquad (\textrm{in Fun}(\Aalg, \s{S})).
\end{displaymath}
Further, we have a lift of the diagram
\begin{displaymath}
  \xymatrix{ \widehat{G} \times \Lambda_1^0  \ar[r] \ar[d] &  \sM_{\widetilde{\QC(X)}_{P/}} \ar[d] \\ \widehat{G} \times \Delta_1 \ar[r] \ar^H@{-->}[ru] & \sM_{\widetilde{\QC(X)}}}
\end{displaymath}
since the right vertical arrow is a left fibration and the left vertical arrow is left anodyne.  Here the bottom arrow corresponds to the natural map $\widehat{G} \times \Delta_1 \rightarrow \Aalg \times \Delta_1  \to \Aalg \xrightarrow{Q} \sM_{\QCF}$.  It is important to note that the morphsim $\Aalg \times \Delta_1 \to \Aalg$ is not the projection, but instead obtained above.  It is easy to see that $H$ is a morphism from $\widehat{G} \rightarrow \widehat{F}$. The fact that $\widetilde{\QC(X)}$ is Cartesian ensures that this is a categorical equivalence.  Combined with the aforementioned categorical equivalence $\mV_{P, Q} \rightarrow \widehat G$, we get a categorical equivalence from $\mV_{P, Q} \rightarrow \widehat{F}$.  This is best seen as a natural equivalence functors: 
\begin{displaymath}
  (B \rightarrow \mr{Hom}_{X[B]}(P\otimes B, Q\otimes B)) \cong (B \rightarrow \mr{Hom}_{X[A]}(P, Q \otimes B)).
\end{displaymath}
\end{proof}

\begin{corollary}
  $\mV_{P, Q}$ is a $(b-a)$-geometric stack.
\end{corollary}
\begin{proof}
  This follows immediately from Proposition \ref{homotopy_coherent} and Proposition \ref{geometric_stack_rep}.
\end{proof}

\begin{proposition}
  $\mI_{P, Q} \subset \mV_{P, Q}$ is a Zariski open substack.
\end{proposition}
\begin{proof}
Let $\widehat F$ fit into the fiber diagram
  \begin{displaymath}
    \xymatrix{\widehat F \ar[r] \ar[d]& \mI_{P, Q} \ar[d] \\ B\textrm{-Alg} \ar[r]^{\phi} & \mV_{P, Q}}
  \end{displaymath}
where $B \in \Aalg$.  Again, we view $\phi$ as a functor and as an object in the fiber of $\mV_{P, Q}$ above $B$.  In particular, $\phi: P\otimes_A B \to Q\otimes_A B$.  We know that $\fib(\phi) \in \pscoh_{\pi}^{[a,b + 1]}$.  Thus, $\pi_0(B)\otimes_{\sO_{X[B]}} \fib(\phi)$ is a bounded above complex of vector bundles with bounded cohomology.  Set $Z := \mr{supp}(\fib(\phi))$; $Z$ is a Zariski closed subset of $X[\pi_0(B)]$ (and thus $X[B]$ as well).  Since $X[B] \rightarrow \spec B$ is projective, $\pi(Z)$ is a Zariski closed subset of $\spec B$. Let $\s{U}$ be the open complement of $\pi(Z)$.  This is a Zariski open subscheme of $\spec B$. We claim $\widehat F \cong \sU$.

For any $g: B \to C$ with $\spec C \to \spec B$ factoring through $\sU$, the morphism $f^\ast \phi \in \mI_{P, Q}$ since restricted to $X\times \sU$, $\fib(\phi|_{\s{U}}) \cong 0$.  Technically, the definition of $\sU$ shows the statement is true for $\phi|_{t_0(\sU)}$, but the morphism $t_0(\s{U}) \rightarrow \spec B$ factors $t_0(\sU) \rightarrow \sU \rightarrow \spec B$.  An element $\sF \in \pscoh_\pi(X \times \sU)$ is such that $(id \times t_0)^\ast \sF \cong 0$ if and only if $\sF \cong 0$ (this is an easy extension of the definition and the similar easy statement regarding perfect complexes).  Thus the statement for $\phi|_{\s{U}}$.  This yields a natural morphism $\sU \rightarrow \widehat F$.  To show this morphism is an equivalence we will need explicit an explicit model of $\widehat F$.  

We assume that the morphism  
\begin{displaymath}
  \sM_{\QCF^{\Delta_1, \sharp}} \rightarrow \sM_{\QCF^{\Delta_1, \flat}}.
\end{displaymath}
is a monomorphism and a fibration in the covariant model structure on $(\sset)_{/\kAlg}$.  This will be explicitly verified at the end of this proof.  This assumption implies $\mI_{P, Q} \rightarrow \mV_{P, Q}$ is a full sub-left fibration and allows for an explicit model for $\widehat{F}$: it is simply the fiber product as a simplicial set over $\Aalg$.   In particular, $\widehat F \rightarrow \Balg$ is a full sub-left fibration of $\Balg$ containing $\sU$. 
Since $f^\ast \phi \in \mI_{P, Q}$ (over the vertex $C$) for $f: \spec C \rightarrow \spec B$ if and only if $\fib(f^\ast \phi)$ vanishes (this follows from standard Tor spectral sequences), the equivalence $\sU \cong \widehat F$ is then an easy consequence of the fact that $\sU$ is universal for vanishing of $\fib(\phi)$ on subschemes of the form $X\times Y$ with $Y \subset \spec B$.   

We now show that \begin{displaymath}
  \sM_{\QCF^{\Delta_1, \sharp}} \rightarrow \sM_{\QCF^{\Delta_1, \flat}}.
\end{displaymath}
is a fibration.  It suffices to show that there exists a lift for any trivial cofibration $f: S \rightarrow K$.  As mentioned above, the moduli functor $\sM_{-}$ has as left quillen adjoint $\sharp: K \in (\sset)_{/\kAlg} \rightarrow K^\sharp$.  A lift in the diagram 
\begin{displaymath}
  \xymatrix{S \ar[r] \ar[d] & \sM_{\QCF^{\Delta_1, \sharp}} \ar[d] \\ K \ar[r] & \sM_{\QCF^{\Delta_1, \flat}} }
\end{displaymath}
corresponds to a lift in the diagram
\begin{displaymath}
  \xymatrix{S^\sharp \ar[r] \ar[d] & \QCF^{\Delta_1, \sharp} \ar[d] \\ K^\sharp \ar[r] & \QCF^{\Delta_1, \flat} }
\end{displaymath}
The top morphism is equivalent to a morphism $S^\sharp \times \Delta_1^\sharp \rightarrow \QCF$ and the bottom morphism is equivalent to a morphism $K^\sharp \times \Delta_1^\flat \rightarrow \QCF$.  Thus the above diagram is then equivalent to a morphism
\begin{displaymath}
S^\sharp \times \Delta_1^\sharp \coprod_{S^\sharp \times \Delta_1^\flat} K^\sharp \times \Delta_1^\flat \rightarrow \QCF
\end{displaymath}
and a lift in the above diagram is a lift diagram
\begin{displaymath}
  \xymatrix{S^\sharp \times \Delta_1^\sharp \coprod_{S^\sharp \times \Delta_1^\flat} K^\sharp \times \Delta_1^\flat \ar[d] \ar[r] & \QCF \ar[d] \\ K^\sharp \times \Delta_1^\sharp \ar@{-->}[ru]\ar[r] & \kAlg }
\end{displaymath}
Since $\sharp$ is Quillen, $f^\sharp$ is a weak cofibration in $(\sset^+)_{/\kAlg}$ and thus marked anodyne.  The morphism $\Delta_1^\flat \rightarrow \Delta_1^\sharp$ is a cofibration, so by \cite[Proposition 3.1.2.3]{topoi}
\begin{displaymath}
  S^\sharp \times \Delta_1^\sharp \coprod_{S^\sharp \times \Delta_1^\flat} K^\sharp \times \Delta_1^\flat \rightarrow K^\sharp \times \Delta_1^\sharp
\end{displaymath}
is marked anodyne and the lift exists.  Thus $\sM_{\QCF^{\Delta_1^\sharp}} \rightarrow \sM_{\QCF{\Delta_1^\flat}}$ is a fibration.  The statement about being a monomorphism follows easily from the fact that $\QCF^{\Delta_1, \sharp} \subset \QCF^{\Delta_1, \flat}$.  



\end{proof}

\subsection{Atlas for $\sMab$}
The following has been proved in multiple places, for instance \cite{KC} in characteristic $0$ or \cite[Example 4.15]{JP} in general.
\begin{proposition}
$\sM_{\bderive{X}}^{[0]}(X)$ is a geometric 1-stack locally of almost finite type.
\end{proposition}


  

Throughout the rest of this section, assume that $\sMab$ is a $(b-a + 1)$-geometric $D^-$-stack of almost finite type.  We will show the case of $[a, b+1]$ by induction.  The idea is to provide an atlas by uniformly writing elements of Tor-amplitude $[a, b+1]$ as an extensions of elements with Tor-amplitude of $[a, b]$.  This is similar to what was done in \cite{dg_moduli}, although the projectivity will somewhat complicate matters.

Define $\Morph_{ab}$ to be a choice of left fibration (over $\kAlg$) that makes a fiber square
\begin{displaymath}
\xymatrix{ \Morph_{ab} \ar[r] \ar[d] & \sM_{\QCF^{\Delta_1}} \ar[d]^{s
    \times t} \\  \sMab \times \sM_{\widetilde{\Vect[-b]}} \ar[r] & \sM_{\QCF} \times \sM_{\QCF}}
\end{displaymath}
Where the right vertical morphism is induced by applying the moduli
functor to 
 \begin{align}
\label{fiber_diagram}
\xymatrix{ \QCF^{\Delta_1} \ar[d]^{s \times t} \ar[r]^{\qquad \mr{fib}} & \QCF \\ \QCF \times \QCF. &  }
\end{align} 

The composition of morphisms
\begin{displaymath}
\xymatrix{ \Morph_{ab}  \ar[r] &\sM_{\QCF^{\Delta_1}}  \ar[r]^{\fib}& \QCF}
\end{displaymath}
 gives a morphism of left fibrations
 \begin{displaymath}
   \Morph_{ab} \rightarrow \sMtor.
 \end{displaymath}
($\Morph_{ab}$ is a left fibration since it is the pullback of one).
By Lemma \ref{relative_tor_prop}, this morphism factors through the natural morphism $\sMabone \rightarrow \sMtor$.  Our assumptions ensure that we can apply Proposition \ref{aligned}; this in turn shows it is an epimorphism on $\pi_0(\sMabone)$. 

\begin{proposition}
\label{geometric_fiber}
$\Morph_{ab}$ has is a $(b-a+1)$-geometric $D^-$-stack locally of almost finite type.
\end{proposition}
\begin{proof}
We will show that the natural morphism $s\times t: \Morph_{ab}
\rightarrow \sMab \times \sM_{\widetilde{\Vect[-b]}}$ is representable.  Given a morphism $(\sF,
\sV): \spec A \rightarrow \sMab \times \sM_{\widetilde{\Vect[-b]}}$, let $\widetilde{H}$ fit into the diagram
\begin{displaymath}
  \xymatrix{\widetilde{H} \ar[r]\ar[d] &\Morph_{ab} \ar[r] \ar[d]&
    \sM_{\QCF^{\Delta_1}} \ar[d] \\ \spec A \ar[r]& \sMab \times \sM_{\widetilde{\Vect[-b]}} \ar[r] &  \sM_{\QCF} \times \sM_{\QCF}}
\end{displaymath}
Since both inside squares are Cartesian, the outside square is Cartesian as well.  Proposition \ref{homotopy_coherent} can be applied to the outside square, showing that  $F_{\widetilde{H}}$ (the straightened form of $\widetilde{H}$ over $\Aalg$) is equivalent to the functor $F: B \rightarrow \Map_{X[A]}^L(\sF, \sV\otimes B)$.  

By Proposition \ref{corepresentable} there exists a $Q$ such that  $F \cong \Map_{A}^L(Q, -)$.  Further, Lemma \ref{itpsuedo} and Proposition \ref{geometric_stack_rep} shows that this functor is equivalent to an affine stack \cite{champs_affine} with geometricity $(b-a)$ and almost finite type. Thus, $\widetilde{H}$ is $(b-a)$-geometric $A$-stack.  
Therefore the morphism $\Morph_{ab} \rightarrow \sMab \times \sM_{\Vect[-b]}$ is $(b-a)$-representable. By \
cite[Proposition 1.3.3.4]{HAGII}, $\Morph_{ab}$ is $(b-a +1)$-geometric locally of almost finite type.
\end{proof}

Ideally, $\fib: \Morph_{ab} \rightarrow \sMabone$ would be smooth, thus allowing the construction of a smooth atlas.  This will not generally be true since $\pi_{\ast}: \QC(X[A]) \rightarrow \QC(A)$ is not exact (in the t-structure).  To fix this problem, we note that $\Morph_{ab} \rightarrow \sMabone$ factors
\begin{displaymath}
  \xymatrix{\Morph_{ab} \ar[r]^{t \times \fib\qquad \qquad} & \sM_{\widetilde{\Vect[-b]}} \times \sMabone \ar[r]& \sMabone}
\end{displaymath}
We restrict this morphism to a Zariski open substack of $\sM_{\widetilde{\Vect[-b]}} \times \sMabone$ for which $t \times \fib$ has the added property of smoothness.  

Let $\s{L}_{ab} \subset \sM_{\widetilde{\Vect[-b]}} \times \sMabone$ be the full sub-left fibration (it is a left fibration by right exactness of pullback) whose fiber over $\Aalg$ consists of pairs $(\s{V}, \s{G}) \in \Vect[-b](X[A]) \times \pscoh_{\pi}^{[a,b+1]}(X[A])$ with $R\mr{Hom}_{X[A]}(\s{V}[-1], \s{G}) \in \QC^{(-\infty, 0]}(A)$. 

\begin{proposition}
\label{Lzar}
$\s{L}_{ab}$ is a $D^-$-stack.  Further, the natural morphism $\sL_{ab} \rightarrow \sM_{\Vect[-b]} \times \sMabone$ is a Zariski open immersion.
\end{proposition}
\begin{proof}
To see that $\s{L}_{ab}$ is a $D^-$ stack, one just needs to verify the conditions of Proposition \ref{locally_determined}.   These follow directly from the definition of our t-structure and the right-exactness of the tensor product.

For the second statement, we will show $\s{Z} := \sL_{ab} \times_{\sM_{\Vect[-b]} \times \sMabone} \Aalg$ is a Zariski open subscheme of $\Aalg$.  Using the notation in the following commutative diagram
\begin{displaymath}
  \xymatrix{X[B] \ar[r]^{f^\prime} \ar[d]^{\pi^\prime} & X[A]\ar[d]^\pi \\ \spec B \ar[r]^f & \spec A}
\end{displaymath}
the definition of $\sL_{ab}$ and using arguements similar to that of Proposition \ref{ab_is_zariski_open}, we can assume $\s{Z} \subset \Aalg$ and   
\begin{align*}
(\s{Z})_0 = \{f \in (\Aalg)_0 |\; \tau^{\geq 1}(R\mr{Hom}_{X[B]}(f^{\prime \ast} \sV, f^{\prime \ast} \s{G})) \cong 0 \} \\
\cong \{f \in (\Aalg)_0 |\; \tau^{\geq 1}(\pi^\prime_{\ast}f^{\prime \ast}(\sV^\vee \otimes \sG)) \cong 0 \} \\ 
\cong \{f \in (\Aalg)_0 |\; \tau^{\geq 1}(f^\ast \pi_\ast (\sV^\vee \otimes \sG)) \cong 0 \}
\end{align*}
Our assumptions on $X$ ensure  $\pi_{\ast}(\sV^\vee \otimes \sG)$ is a perfect complex in $\QC(A)$.  By Proposition \ref{semicontinuous} (in the next subsection) this substack is a Zariski open subscheme in $\Aalg$.


\end{proof}

\subsubsection{Semi-continuity}

As used in the previous proposition, we need a version of semi-continuity for perfect objects over $\spec A$ where $A$ is a simplicial commutative ring.  This is well known and a simple extension of analogous statements in \cite{SGA6}.  

Recall from \cite{SGA6} that given a locally ringed topos $(X, \s{O}_X)$ and a perfect complex in $M \in \QC(X)$ associated with any integer $i$ we can define $rk^M_i(x)$ to be the function $x \rightarrow \rank(H^i(M_x \otimes_{\sO_X,x} k(x)))$, for any point $x$ of $(X, \sO_X)$. 

Associated to any $A \in \kAlg$, we have the locally ringed $\infty$-topos $(\spec A, A)$.  The definition of $rk^M_i$ carries over unaltered to this case since the morphism $A_x \rightarrow k(x)$ factors through $A_x \rightarrow \pi_0(A_x)$.  See \cite[\S 4.2]{DAGV} for the precise definition of $(\spec A, A)$.


\begin{lemma}
\label{semicontinuous}
Let $M \in \Perf(A)$.  Then $rk^M_i(x)$ is upper semi-continuous in the locally ringed $\infty$-topos $(\spec A, A)$.
\end{lemma}
\begin{proof}
Since the rank of a projective object is invariant under passage to $\pi_0(A)$, one easily sees that the statement is ``topological'' (as used in \cite{lurie_thesis}) and thus we can replace $(\spec A, A)$ by $(\spec A, \pi_0(A))$.  Further, the equivalence between the category of etale $A$-algebras and etale $\pi_0(A)$-algebras restricts to an equivalence on Zariski open sets over $A$ and $\pi_0(A)$.  Thus, the problem is reduced to the classical case and follows from \cite[Lemma 5.5]{SGA6}.
\end{proof}
\subsubsection{The substack $\s{A}_{ab} \subset \Morph_{ab}$}
\label{funny_thing}

Finally, we define $\s{A}_{ab}$ to be the left fibration fitting into the fiber diagram
\begin{displaymath}
  \xymatrix{ \s{B}_{ab} \ar[r] \ar[d] &\Morph_{ab}\ar[d]^{t \times \fib} \\ \s{L}_{ab} \ar[r] &  \sM_{\Vect[-b]} \times \sMabone}
\end{displaymath}
The above proposition shows the natural morphism $\sL_{ab} \rightarrow \sM_{\Vect[-b]} \times \sMabone$ is a Zariski open immersion.  Thus, the $(b-a + 1)$-geometricity of $\Morph_{ab}$ carries over to $\s{A}_{ab}$ as well and we have shown

\begin{corollary}
  $\s{A}_{ab}$ is a $(b-a +1)$-geometric $D^-$-stack locally of almost finite type.
\end{corollary}

Before showing that $\s{A}_{ab} \to \sL_{ab}$ is smooth, we need to refine the above diagram and give a more explicit description of $\s{A}_{ab}$.  Let $\s{E} \subset \widetilde \QC(X)^{\Delta_2 \times \Delta 1, \flat}$ be the full subcategory consisting of distinguished triangles, i.e., diagrams of the form
\begin{displaymath}
  \xymatrix{\sV \ar[r] \ar[d] & \sG \ar[r] \ar[d] & 0^\prime \ar[d] \\ 0 \ar[r] & \s{W} \ar[r] & \sV^\prime} 
\end{displaymath}
satisfying
\begin{enumerate}
\item all objects are in $\QC(X[A])$ for some $A \in \kAlg$,
\item $0, 0^\prime$ representatives of the zero object,
\item $\sV \cong \sV^\prime[-1]$
\item both squares are Cartesian.
\end{enumerate}

Recall from \cite[Remark 1.1.1.7]{higher_algebra}, restricting the natural fibrations 
\begin{displaymath}
  \xymatrix{ &  \QCF^{\Delta_2 \times \Delta_1, \flat} \ar[rd]^{\phi_1^\ast} \ar[ld]_{\phi_2^\ast} & \\ \QCF^{\Delta_1, \flat} & & \QCF^{\Delta_1, \flat}}
\end{displaymath}
obtained from the cofibrations $\phi_1: \Delta_1 \rightarrow \Lambda^{0}_2 \times {1} \subset \Delta_2 \times \Delta_1$ and $\phi_2: \Delta_1 \rightarrow {1} \times \Delta_1 \subset \Delta_2 \times \Delta_1$ to $\sE$ are trivial fibrations. This is just repeated use of \cite[Proposition 4.3.2.15]{topoi}.  The statement remains true in the relative case since coCartesian fibrations are categorical fibrations.  If $(\phi_{\sE})_{1\ast}$ denotes a choice of section of $\phi_1^\ast$, the morphism $\fib$ in diagram (\ref{fiber_diagram}) is $\phi_{2}^\ast|_{\sE} \circ (\phi_{\sE})_{1\ast}$ composed with the fibration $\QCF^{\Delta_1} \rightarrow \QCF$ given by $\Delta_0 \cong \Lambda_{1}^0 \subset \Delta_1$.  Since $\Morph_{ab}$ is a full sub-left fibration of $\sM_{\QCF^{\Delta_1}}$, to remove choices we replace $\QCF^{\Delta_1}$ with $\sE$ and $\Morph_{ab}$ with the equivalent full sub-left fibration of $\sM_{\sE}$ consisting of triangles with $\s{W} \in \pscoh_{\pi}^{[a, b]}(X[A])$ and $\sV^\prime \in \Vect[-b](X[A])$. 

If we define $\Morph_{\sL_{ab}} \subset \sM_{\sE}$ to be the full sub-left fibration consisting of triangles with $(\sV^\prime, \sG) \in \sL_{ab}$ then $\Morph_{\sL_{ab}} \cong \sL_{ab} \times_{\sM_{\QCF} \times \sM_{\QCF}} \sM_{\sE}$ and we have a diagram
\begin{displaymath}
  \xymatrix{\s{A}_{ab} \ar[r] \ar[d] & \Morph_{ab} \ar[d] \\ \Morph_{\sL_{ab}} \ar[r] \ar[d]^f & \sM_{\sE} \ar[d]^{|_{(0,0) \coprod (0,1)}} \\ \sL_{ab} \ar[r] & \sM_{\QCF} \times \sM_{\QCF}}
\end{displaymath}
with all squares Cartesian.  In other words, $\s{A}_{ab}$ is the full sub-left fibration consisting of triangles with $((\sV^\prime, \sG), \s{W}) \in \sL_{ab} \times \sMab$.  Restricting distinguished triangles to $\s{W}$ induces a fibration $\Morph_{\sL_{ab}} \xrightarrow{|_{(2,1)}} \sM_{\QCF}$ with $\Morph_{\sL_{ab}} \cong \Morph_{\sL_{ab}} \times_{\sM_{\QCF}} \sMabone$ and  $\s{A}_{ab} \cong \Morph_{\sL_{ab}} \times_{\sM_{\QCF}} \sMab$ with the morphism $\s{A}_{ab} \to \Morph_{\sL_{ab}}$ (in the diagram above) induced by the Zariski open immersion $\sMab \subset \sMabone$.  Thus, $\s{A}_{ab}$ is a Zariski open substack of $\Morph_{\sL_{ab}}$.

\begin{proposition}
\label{smooth}
The morphism $\s{A}_{ab} \rightarrow \sL_{ab}$ is smooth.
\end{proposition}
\begin{proof}

This is basically a restatement on the conditions for an pair of quasi-coherent sheaves to belong to $\s{L}_{ab}$. 
Using the same methods as in Proposition \ref{geometric_fiber}, the the fiber $\Aalg \times_{\sL_{ab}} \Morph_{\sL_{ab}}$ of the morphism $f$ (using the notation in the diagram above) over any $(\sV^\prime, \sG) \in \sL_{ab}(A)$ is the left fibration corresponding to the functor $\Map_{X[A]}^L(\sV, \sG \otimes \pi^\ast -)$.  By Proposition \ref{corepresentable}, this functor is equivalent ot $\Map^L(Q, -)$ for some $Q \in \pscoh(A)$.  Since $(\sV^\prime, \sG) \in \sL_{ab}$, $R\mr{Hom}^i(\sV^\prime[-1], \sG\otimes -) = 0$ for $i > 0$ (this follows from proper base change), implying that $\Map^L(Q, K) \cong 0$ if $K \in \QC^{(-\infty, -1]}(A)$.  It is clear that $Q \cong (\pi_{\ast} (\sV^\vee \otimes \sG))^{\vee}$, and thus $Q$ is perfect with Tor-amplitude $[0, b-a + 1]$. By \cite[3.9]{dg_moduli}, this is smooth.   

From the discussion above, $\sA_{ab} \to \Morph_{\sL_{ab}}$ is a Zariski open substack, in particular, it is smooth.  Thus the natural morphism $\sA_{ab} \to \sL_{ab}$ is the composition of two smooth morphisms, and thus smooth.

\end{proof}

\subsubsection{The atlas of $\sMabone$}

Let $L$ be a choice of ample bundle on $X$ (which exists by assumption).  Let $\iota_{n,m}: \sMabone \rightarrow  \sM_{\Vect[-b]} \times \sMabone$ denote the natural section given by the vector bundle $\oplus^m L^{\otimes n}[-b]$.  The morphism $\iota_{n,m}$ induces the fiber diagram
\begin{displaymath}
 \xymatrix{
 \iota_{n,m}^{\ast} \s{A}_{ab} := \s{A}_{ab} \times_{\sM_{\Vect[-b]} \times \sMabone} \sMabone \ar[r] \ar[d] & \s{A}_{ab} \ar[d] \\
 \sL_{ab} \times_{ \sM_{\Vect[-b]} \times \sMabone} \sMabone \ar[r] \ar[d]& \s{L}_{ab}\ar[d] \\
\sMabone \ar[r]^{\iota_{n,m}} & \sM_{\Vect[-b]} \times \sMabone }
\end{displaymath}
Proposition \ref{smooth} shows that the top right vertical morphism is smooth, while Proposition \ref{Lzar} shows the bottom right vertical morphism is smooth.  Smoothness is stable under base change, thus the composition of the left vertical morphisms is smooth.  Denoting this composed morphism as $s_{n,m}$, we will show that $s := \coprod_{n,m} s_{n,m}: \coprod_{n,m} \iota_{n,m}^\ast \s{A}_{ab} \to \sMabone$ is an epimorphism.

\begin{proposition}
$\coprod_{n,m} \iota^\ast_{n.m}\s{A}_{ab} \rightarrow \sMabone$ is an epimorphism.
\end{proposition} 
\begin{proof}

Given a morphism $\sF: \spec A \rightarrow \sMabone$, we find a Zariski cover $U_{N,\alpha}$ of $\spec A$ with $\sF|_{X \times U_{N, \alpha}}$ in the image of $s$.  Per our usual convention, when no confusion can arise, the natural morphism $X\times Y \to Y$ will be denoted as $\pi$, even when $Y$ changes.  For simplicity, we assume that $b = -1$ and write $\sF(n)$ for $\sF \otimes L^{\otimes n}$.  With these assumptions $H^{0}(\sF)$ is the highest non-zero cohomology.  Define $U_{N, n}$ as the maximal Zariski open subset of $\spec A$ having the properties
\begin{enumerate}
\item $\pi_\ast(\sF(n)|_{X \times U_{N,n}}) \in \QC^{(-\infty, 0]}(X \times U_{N, n})$,
 \item there exists a surjection $\mu_{N, n}: \oplus_{N} \sO_{\spec U_{N,n}} \twoheadrightarrow H^0(\pi_{\ast}(\sF(n)|_{X \times U_{N, n}}))$.
\item \label{thethird} $H^0(\eta): H^0(\pi^\ast \pi_\ast( \sF(n)|_{X \times U_{N,n}})) \rightarrow H^0(\sF(n)|_{X \times U_{N,n}})$ is surjective.
\end{enumerate}
where $\eta$ is the natural adjuction $\pi^\ast \pi_\ast (\sF(n)|_{X \times U_{N,n}}) \to \sF(n)|_{X \times U_{N,n}}$. These are Zariski open conditions by semi-continuity and the fact that the support of a finitely generated module is Zariski closed.

We assume the existence of a $n$ such that $\spec A = \cup_N U_{N, n}$ (this will be verified below).  It is clear that we can refine this cover to an affine cover $\{\spec B_{N,\alpha}\}$.  
 We have the following isomorphisms:
\begin{align*}
H^0(\pi_\ast(\sF(n)|_{X[B_{N,\alpha}]})) & \cong R\mr{Hom}^0_{\spec B_{N, \alpha}}(B_{N, \alpha}, \pi_\ast (\sF(n)|_{X[B_{N, \alpha}]}))\\ & \cong R\mr{Hom}^0_{X[B_{N,\alpha}]}(\pi^\ast B_{N, \alpha}, \sF(n)|_{X[B_{N, \alpha}]})\\ & \cong R\mr{Hom}^0_{X[B_{N, \alpha}]}(\sO_{X[B_{N,\alpha}]}, \sF(n)|_{X[B_{N, \alpha}]})
\end{align*}
As consequences of these isomorphisms, pulling back $\mu_{N, \alpha}$, we have 
\begin{displaymath}
\pi^\ast (\mu): \oplus_{N} \sO_{X[B_{N, \alpha}]} \twoheadrightarrow  R\mr{Hom}^0_{X[B_{N, \alpha}]}(\sO_{X[B_{N,\alpha}]}, \sF(n)|_{X[B_{N, \alpha}]})\otimes_{B_{N, \alpha}} \sO_{X[B_{N, \alpha}]}
\end{displaymath}
with $H^0(\pi^\ast \mu)$ a surjection.  Further, the natural morphism $e: R\mr{Hom}^0(\sO_{X[B_{N,\alpha}]}, \sF(n)|_{X[B_{N, \alpha}]})\otimes \sO_{X[B_{N.\alpha}]} \rightarrow \sF(n)|_{X[B_{N, \alpha}]}$ factors through a the natural co-unit morphism $\eta$ and this factorization is a surjection on $H^0$ as well.
Since the composition of two surjections is a surjection, the morphism $\eta \circ \mu$ is such that $H^0(\eta \circ \mu)$ is a surjection.

  The standard cohomological long exact sequence, combined with Lemma \ref{relative_tor_prop},  shows $\cofib(e)$ is pseudo-coherent and has relative Tor-amplitude of $[a, 0]$ with top non-zero cohomology in degree $-1$.  The graded Tor-spectral sequence shows that this must be relative Tor-amplitude of $[a, -1]$.   Since $R\mr{Hom}_{X[B_{N, \alpha}]}^i(\sO_{X[B_{N.\alpha}]}, \sF(n)|_{X[B_{N, \alpha}]}) \cong R\mr{Hom}^i_{\spec B_{N, \alpha}}(B_{N, \alpha}, \pi_\ast (\sF(n)|_{X[B_{N, \alpha}]}))$, property (1) ensures that $(\oplus^N \sO_{X[B_{N, \alpha}]}, \sF(n)|_{X[B_{N, \alpha}]})$ is in $\sL_{ab}$ and thus the evaluation morphism is in $\s{A}_{ab}$.

Tensoring $\QC(X[B_{N, \alpha}])$ by $L^{-n}|_{X[B_{N, \alpha}]}$ is an autoequivalence (of the $\infty$-category), this shows
\begin{displaymath}
\sF|_{\spec B_{N, \alpha}} \cong \fib(\cofib(e) \rightarrow (\oplus^N L^{-n}|_{X[B_{N, \alpha}]})[1])
\end{displaymath}
In other words, restricted to the open cover $\spec B_{N, \alpha}$, $\sF$ is in the image of the morphism $\coprod_{N} s_{n,N}: \coprod_{N} \iota_{-n,N}^\ast \s{B}_{ab} \rightarrow \sMabone$.  Thus the proposition is proved.  

We now show the existence of an $n$ such that $\cup_N U_{N, n} \cong \spec A$.  That such a cover exists that satifies the first property is a direct consequence of our definition of projectivity and proper base change.  
For the second property, we note by Lemma \ref{strong_dual}, $\pi_\ast \sF$ is perfect.   Since $\spec A$ is affine, it is strict perfect and there exists a projective $\mathfrak{p} \to \pi_\ast \sF$ surjective on $H^0$, and the second property easily follows. 

The last property is well known when $A$ is discrete.  The general case follows from the discrete case as follows.  There is a one-to-one correspondence between Zariski open subschemes of $\spec A$ and $\spec \pi_0(A)$.  We claim the covering  and $n$ for $\spec \pi_0(A)$ will suffice for $\spec A$, via this correspondence.  To ease notation, let $\bar A := \pi_0(A)$ and assume property (\ref{thethird}) holds for $\spec \bar A$.  Since $X[A] \to \spec A$ is assumed to be flat we have the Cartesian diagram
\begin{displaymath}
  \xymatrix{ t_0(X) \ar[r]_{\iota_X} \ar[d]^{\pi_{\bar A}} & X \ar[d]^{\pi_A}\\ \spec \bar A \ar[r]^{\iota_A} & \spec A} 
\end{displaymath}
where $\iota_X$ and $\iota_A$ are the natural induced adjunctions. With the notation in the diagram, we have $\iota_X^\ast \pi_A^\ast \pi_{A\ast} \sF \cong \pi_{\bar A}^\ast \iota_A^\ast \pi_{A\ast} \sF \cong \pi_{\bar A}^\ast \pi_{\bar A\ast} \iota_X^\ast \sF$.  Since our assumptions ensure $\pi_A^\ast \pi_{A\ast} \sF \to \sF$ is surjective on $H^0$ if and only if the same is true of $\iota_X^\ast \pi_A^\ast \pi_{A\ast} \sF \to \iota_X^\ast \sF$, we have proved our claim.

\end{proof}

To finish, we have now constructed a smooth epimorphism from a $(b-a +1)$-geometric stack locally of finite type onto $\sMabone$.  An atlas of $\s{B}_{ab}$ will then provide the necessary atlas for $\sMabone$.

\section{The motivic derived Hall algebra.}

We assume that $X$ satisfies the conditions of Theorem \ref{itworks} and that $k$ is a discrete $\mZ$-algebra.  We use our explicit verification of the geometricity of $\sMtor$ to define a derived Hall algebra associated to $X$.  Restricting our fibration $\sMtor$ to the full subcategory of discrete $k$-algebras yields a geometric stack of finite type.  We will show below that this stack is locally special, allowing a well behaved Grothendieck ring of special Artin stacks over $\sMtor$. This ring, denoted $\sH(\sMtor)$  is generated by special Artin stacks subject to ``scissor'' relations \cite{Toen_grothendieck}.  

For this section, we will need a relative version of the $\mr{Gap}^0([n],\sC)$ construction in \cite[Remark 11.4]{DAGI}.  In particular, if $[n]$ is the standard linear order we have the ``arrow'' quasi-category $N([n])^{\Delta_1}$.  We denote by $\sE_n$ the full sub-left fibration of $\sM_{\bderive{X}^{N([n])^{\Delta_1}, \flat}}$ such that restricted to the fiber over $A \in \kAlg$ is an element of $\mr{Gap}^0([n], \bderive{X[A]})$. That this is a left fibration is due to the morphisms $f^\ast: \bderive{X[A]} \to \bderive{X[B]}$ being exact for any $f: A \to B$.

The $\s{E}_n$ assimilate into a simplicial object of $(\sset)_{/\kAlg}$, level-wise fibrant.     $d_n^i: \sE_n \to \sE_{n-1}$ will correspond to the $i$th face map.  One has two diagrams

\begin{align}
\label{itcommutes}
  \xymatrix{ \sE_3 \ar[r]^{d_3^1} \ar[d]^{d_3^3 \times d_2^0 \circ d_3^1} & \sE_2 \ar[d]^{d_2^2 \times d_2^0} \ar[r]^{d_2^1} & \sE_1 &  & \sE_3 \ar[r]^{d_3^2} \ar[d]^{d_2^2 \circ d_3^2 \times d_3^0} & \sE_2 \ar[d]^{d_2^2 \times d_2^0} \ar[r]^{d_2^1} & \sE_1 \\
    \sE_2 \times \sE_1 \ar[r]^{d_2^1 \times id} \ar[d]^{d_2^2 \times d_2^0 \times id} & \sE_1 \times \sE_1 & & & \sE_1 \times \sE_2 \ar[r]^{id \times d_2^1} \ar[d]^{id \times d_2^2 \times d_2^0} & \sE_1 \times \sE_1 & \\
    \sE_1 \times \sE_1 \times \sE_1 & & &  &  \sE_1 \times \sE_1 \times \sE_1 & &  }
\end{align}
In both diagrams the vertical composite $\sE_3 \to \sE_1 \times \sE_1 \times \sE_1$ is induced from the product of restrictions arising from the functors $[1] \rightarrow [3]$ sending $0 \to i, 1 \to i+1$, $i = 0, 1, 2$.  The horizontal composite is the induced restriction map arising from the functor $[1] \to [3]$ sending $0 \to 0, 1 \to 3$.


\begin{proposition}
\label{so_special}\mbox{}
  \begin{enumerate}
  \item $\sE_1 \cong \sMtor$ and is locally special.
  \item $\sE_2 \cong \sM_{\bderive{X}^{\Delta_1, \flat}}$ and is locally special.
  \item $d_2^2 \times d_2^0$ is strongly of finite type. 
  \item The commutative squares in Diagram (\ref{itcommutes}) are fiber squares.
 \item $d_3^3 \times (d_2^0 \circ d_3^0)$ and $d_3^0 \times (d_2^2 \circ d_3^3)$ are strongly of finite type.
 \item $\sE_3$ is locally special 
   \end{enumerate}
\end{proposition}
\begin{proof}

{\it1)}  The first statement is an extension of \cite[Lemma 11.3]{DAGI} to the relative case.  This extension is possible since the cited proof of the lemma is repeated application of \cite[Proposition 4.3.2.15]{topoi}.  The arguement goes through unchanged since we can apply this proposition to our case as well (using the fact that coCartesian fibrations are categorical fibrations).  

For the second statement, it suffices to show $\sMab$ is special.  Let $K$ be a field and $\sF: \Kalg \to \sMab$ correspond to an object $\sF \in \pscoh^{[a,b]}_\pi(X[K])$.  The fibration $\widetilde G := \Kalg \times_{\kAlg} \sMab$, considered as a fibration over $K$-Alg has a natural section $s_{\sF}$ given by $\sF$.  We must show that the resulting sheaves $\pi_n(\widetilde{G}, s_{\sF})$ are representable by group schemes over $\spec K$ for $i > 0$ and are unipotent for $i > 1$. 

Using that $\mI_{\sF, \sF}$ is the (pointed) loop space of $\widetilde{G}$ (as stacks over $\spec K$) and $\mI_{\sF, \sF}$ is a Zariski open substack of $\mV_{\sF, \sF}$, we know that $\pi_i(\widetilde G, s_{\sF}) \cong \pi_{i-1}(\mV_{\sF, \sF})$ as group sheaves for $i > 1$ and a Zariski open immersion (as sheaves) for $i = 1$.   Since $\mV_{\sF, \sF}$ is corepresented by $\Map_{\Kmod}^L(Q, -)$ with $Q \in \pscoh(K)$, by \ref{geometric_stack_rep} it is an affine stack. The result then follows from  \cite[Theorem 2.4.5]{champs_affine} (although we are not in the connected case, since we are dealing with $i > 0$, we can reduce to this case). 

{\it 2)}  The first statement proceeds exactly as above.   To show $\sM_{\widetilde{\bderive{X}}^{\Delta_1, \flat}}$ is locally special, let $K$ be a field with structure morphism $f: \spec K \to \spec k$ and $f^*: St_{k} \to St_{K}$ be the natural functor given by precomposition with the forgetful functor $K\textrm{-Alg} \to \kAlg$. If $m: \spec K \to \sM_{\widetilde{\bderive{X}}^{\Delta_1, \flat}}$ a morphism with source $\sF$ and target $\sG$ and $\sF, \sG \in \pscoh_{\pi_K}^{[a.b]}(X)$, one has a diagram of pointed stacks:
\begin{displaymath}
       \xymatrix{  & & f^\ast \sM_{\widetilde{\bderive{X}}^{\Delta_1, \flat}} \ar[d]^{s \times t} \\ \ast \ar[rru]^{s_m} \ar[rr]^{s_{\sF \times \sG}\qquad} & &f^\ast \sMtor \times f^\ast \sMtor}
\end{displaymath}
Taking the (homotopy) fiber of this morphism, $\sH := \fib (s \times t)$ is a pointed geometric stack and there is a long exact sequence of fundamental group sheaves
\begin{displaymath}
  \xymatrix{ \ldots\ar[r]& \pi_i(\sH, s) \ar[r]&  \pi_i(f^\ast \sM_{\widetilde{\bderive{X}}^{\Delta_1, \flat}}, s_m) \ar[r] & \pi_i(f^\ast (\sMtor \times \sMtor) , s_{\sF \times \sG}) \ar[r] & \ldots}
\end{displaymath}

The category of affine group schemes over $K$ and the subcategory of unipotent affine group schemes is closed under extensions and kernels \cite{unipotent}.  From above, we know $\pi_i(f^\ast (\sMtor \times \sMtor) , s_{\sF \times \sG})$are group schemes for $i \geq 1$ (unipotent for $i > 1$).  Thus it is enough to show the same for $\pi_i(\sH, s)$, $i \geq 1$.  Recall that $f^\ast$  has a left adjoint and thus preserves fiber diagrams.  This yields the following diagram with all squares Cartesian:
\begin{displaymath}
  \xymatrix{ \sH \ar[r] \ar[d] & f^* \mV_{\sF, \sG} \ar[d] \ar[r] & f^* \sM_{\widetilde{\bderive{X}}^{\Delta_1, \flat}} \ar[d]^{s \times t} \\ \spec K \ar[r] & f^! \spec K  \ar[r]^{\sF \times \sG\qquad} & f^\ast(\sMtor \times \sMtor)}
\end{displaymath}
The stack $f^\ast \spec K$ is equivalent to a sheaf since have restricted to discrete $K$-algebras.  Thus, one has $\pi_i(\sH, m) \cong \pi_i(f^\ast \mV_{\sF, \sG}, s_m)$ for $i \geq 1$, showing the result (again by \cite[Theorem 2.4.5]{champs_affine}).
 
{\it 3)} The above equivalences (first statements) show that $d_2^2 \cong s$, $d_2^1 \cong t$ and $d_2^0 \cong \cofib \cong \fib[1]$.   So, $d_2^2 \times d_2^0 \cong s \times \fib[1]$.  
Recall from \S \ref{funny_thing}, one has the moduli of distinguished triangles $\sM_{\sE}$ and the equivalence $\phi_1^\ast$.   Using methods similiar to those presented in \S \ref{funny_thing}, $\sE_2$ is equivalent to the full sub-left fibration of $\sM_{\sE}$ consisting of distinguished triangles with all objects in $\widetilde{\bderive{X}}$.  Replacing $\sE_2$ with this sub-left fibration, the morphism $d_2^2 \times d_2^0$ is naturally equivalent to $([-1] \times id) \circ (t \times s) \circ \phi_1^\ast$.   The first and last morphisms are equivalences and thus we are left showing $t \times s$ is strongly of finite type.   However, this follows by the same method as given in Proposition \ref{geometric_fiber}.

{\it 4)} This fact is well known, so we will be brief.  The nerve functor commutes with (homotopy) limits.  Thus, the fiber product $\sC := N([2])^{\Delta_1} \times N([1])^{\Delta_1} \times_{N([1])^{\Delta_1} \times N([1])^{\Delta_1}} N([2])^{\Delta_1}$ can be carried out as categories first.  With this description, $\sC \subset N([3])^{\Delta_1}$ is the largest full subcategory not containing the object $[1\to3]$.  Applying \cite[Proposition 4.3.2.15]{topoi} to the the induced restriction $\sE_3 \to \sE_2 \times \sE_1 \times_{\sE_1 \times \sE_1} \sE_2$ gives an equivalence.   Similar statements (using the left Kan extension, rather than right Kan extension) show the result for the other case.


{\it 5)} Using  {\it 4)}, $d_3^3 \times (d_2^0 \circ d_3^0)$ is the base change morphism of a morphism strongly of finite type.  Since strongly of finite type is stable under base change, the result follows. 

{\it 6)} This follows from special being stable under strongly of finite type base change.
\end{proof}

\begin{remark}
Under the above equivalence $\sE_2 \cong \sM_{\bderive{X}^{\Delta_1, \flat}}$ and $\sE_1 \cong \sMtor$, $d_2^0 \cong \cofib$, $d_2^1 \cong t$ and $d_2^2 \cong s$.  
\end{remark}

\begin{theorem}
Let $X$ be a projective scheme flat over a commutative ring $k$.  Then $\sH(\sMtor)$ has a motivic derived Hall algebra product. 
\end{theorem}

\begin{proof}
This follows closely with the arguements given in \cite[Section 3.3]{toen_overview}.  By Proposition \ref{so_special} and \cite[Example 3.3]{toen_overview} there exists a Grothendieck ring of special geometric stacks, $\sH(\sMtor)$.  By Proposition \ref{so_special}, $s \times \cofib$ is strongly of finite presentation and $\sM_{\widetilde{\bderive{X}}^{\Delta_1, \flat}}$ is locally special.  Thus, using the isomorphism $\sH(\sMtor) \times \sH(\sMtor) \cong \sH(\sMtor \times \sMtor)$ one has a pullback ring homomorphism $\sH(\sMtor) \times \sH(\sMtor)  \to \sH(\sMtor \times \sMtor) \to \sH(\sM_{\widetilde{\bderive{X}}^{\Delta_1, \flat}})$.  Composing this with the natural morphism $\sH(\sM_{\widetilde{\bderive{X}}^{\Delta_1, \flat}}) \xrightarrow{t_!} \sH(\sMtor)$ gives the Hall algebra convolution product: $t_!\circ  (s \times \cofib)^\ast: \sH(\sMtor) \times \sH(\sMtor) \to \sH(\sMtor)$.   We must show this product is associative.

Recall that for a fiber square
\begin{align*}
  \xymatrix{ X \ar[r]^{f^\prime} \ar[d]^{g^\prime} & Y \ar[d]^g \\ Z \ar[r]^f & W} 
\end{align*}
of locally special geometric stacks with vertical morphisms strongly of finite type, $g^\ast \circ  f_! \cong f_!^\prime \circ g^{\prime \ast}$ (as morphisms between $\sH(Z)$  to $\sH(Y)$) .  
By Proposition \ref{so_special}(4), we can apply this to see
\begin{align*}
  (\sH(\sMtor) \ast \sH(\sMtor)) \ast \sH(\sMtor) & \cong d_{2!}^1\circ (d_2^2 \times d_2^0)^\ast \circ (d_2^1 \times id)_! \circ (d_2^2 \times d_2^0 \times id)^\ast \\ & \cong  d_{2!}^1 \circ d_{3!}^1 \circ (d_3^3 \times d_2^0 \circ d_3^1)^\ast \circ (d_2^2 \times d_2^0 \times id)^\ast 
\end{align*}
Likewise, 
\begin{align*}
 \sH(\sMtor) \ast (\sH(\sMtor) \ast \sH(\sMtor)) & \cong d_{2!}^1\circ (d_2^2 \times d_2^0)^\ast \circ (id \times d_2^1)_! \circ (id \times d_2^2 \times d_2^0)^\ast\\ & \cong  d_{2!}^1 \circ d_{3!}^2 \circ (d_2^2 \circ d_3^2 \times d_3^0)^\ast \circ (id \times d_2^2 \times d_2^0)^\ast
\end{align*}
From the remark following diagrams (\ref{itcommutes}), the bottom equations are the same morphism, and thus associativity is shown.

\end{proof}
\appendix
\section{Key properties of psuedo-coherence}

We have included this material as an appendix since many of these proofs are obtained by porting arguments from \cite{TT}.    As mentioned in \S2, depending on the context we will use either $H^i$ or $\pi_{-i}$ to represent the cohomology groups of a element in $\QC(X)$.  Given an filtered object $H$, we will use $V_{H, (k, j)}$ for a choice of $\cofib(H(k) \to H(j))$ , where $k < j$, and $V_{H, j}$ when $k = j-1$. Further notation can be found in \S \ref{notation}.

\begin{proposition}
\label{conversion}
  If $\sF$ is strict $n$-pseudo-coherent and $F$ is a choice of filtration as in Definition \ref{nspc}, then the natural morphism $F(m) \rightarrow \sF$ is a $-m$-quasi-isomorphism for all $m \leq n$.   
\end{proposition}
\begin{proof}
Let $m \leq n$.  The standard t-structure satisfies the conditions of \cite[\S11]{DAGI}.  Thus we have the converging spectral sequence  
\begin{displaymath}
E_r^{pq} = \operatorname{image}(\pi_{p+q}V_{F, (p -r, p)} \rightarrow \pi_{p+q}V_{F, (p -1, p+r-1)}) \rightarrow \pi_{p+q}(\s{F})
\end{displaymath}
Setting $r = 1$, by Definition \ref{nspc}, $\pi_{p + q} V_{F, (p-1, p)} \cong 0$ for $q < 0$.   If 
\begin{displaymath}
  F^m(j) := \begin{cases} F(j) & j \leq m \\ F(k) & j >  m \end{cases}
\end{displaymath}
then $F^m$ is a filtered object with $\varinjlim_\mZ F^m \cong F(m)$ and there exists a natural morphism of filtered objects $F^k \to F$.   The above observation shows
\begin{displaymath}
  (E_r^{p,q})_{F^m} \cong (E_r^{p,q})_F \textrm{ for } \begin{cases} r \in \mathbb{N} & p + q < m \\ r = 1 & p + q = m \end{cases} 
\end{displaymath}
This is enough to show the desired result.
\end{proof}

\begin{corollary}
\label{spcspcalt}
If $\sF$ is strict pseudo-coherent on $\spec A$, then there exists a $P \in \Perf(A)$ such that $P$ has Tor-amplitude with lower bound $n$ and a  $n$-quasi-isomorphism $\phi: P \rightarrow \sF$, i.e., such that $\pi_k(P) \rightarrow \pi_k(\sF)$ is an isomorphism for $k < n$ and surjective for $k = n$.
\end{corollary}
\begin{proof}
This is just the previous proposition with the observation that each $F(n)$ is perfect (being built by a finite number of steps with a finite number of projectives), this gives the result. 
\end{proof}

\begin{lemma}
\label{pullbacka}
If $\sF$ is pseudo-coherent on $X$, then $\iota_X^\ast(\sF)$ is pseudo-coherent on $t_0(X)$.
\end{lemma}
\begin{proof}
It suffices to work locally, and thus with strict pseudo-coherent objects on affine schemes.  In this case, for any $A \in \kalg$, $\iota_A^\ast A \cong \pi_0(A)$ and $\iota_A^\ast$ maps $Perf(A)$ to $Perf(\pi_0)$.    If $\sF$ is a strict pseudo-coherent object, let $F$ be a associated filtered object.  We can form a filtered object associated to $\iota_A^\ast (\sF) \cong \sF\otimes_A \pi_0(A)$ by $F^\prime(n) = F(n)\otimes_A \pi_0(A)$. By Proposition \ref{absolute_tor}, $F(n)$ and $\iota_A^\ast F(n)$ have the same Tor-amplitude.  Since pullback commutes with colimits, $F^\prime$ satisfies all the conditions necessary for $\iota_A^\ast(\sF) \in \pscoh(\pi_0(A))$.
\end{proof}

\begin{lemma}\mbox{}
\label{portingTT}
  \begin{enumerate}
  \item Let $\sF \to \sG \to \sH$ form a distinguished triangle, then the following are true
    \begin{enumerate}
    \item If $\sF$ is $n+1$-psuedo-coherent and $\sG$ is $n$-psuedo-coherent, then $\sH$ is $n$-psuedo-coherent.
    \item If $\sG$ and $\sH$ are $n$-psuedo-coherent, then so is $\sG$
    \item If $\sG$ is $n+1$-psuedo-coherent and $\sH$ is $n$-psuedo-coherent, then $\sF$ is $n+1$-psuedo-coherent.
    \end{enumerate}
  \item $\sF \oplus \sG$ is $n$-psuedo-coherent if and only if $\sF$ and $\sG$ are.
  \end{enumerate}
\end{lemma}
\begin{proof}
{\it 1)} By shifting triangles, it suffices to prove the first assertion.  The question is local, so we can assume $X$ is affine and $\sF, \sG$ are strict pseudo-coherent.  Using the methods outlined in \cite[Lemma 2.2]{EKMM}, or using an extension of \cite[Lemma 1.9.4]{TT} we can assume that the morphism $\sF \to \sG$ is cellular: i.e., there exists a filtrations $F, G$ of $\sF$ and $\sG$(respecitively) such that we have the following commutative diagram for all $k$.
\begin{displaymath}
  \xymatrix{ F(k) \ar[r] \ar[d] & G(k) \ar[d] \\ \sF \ar[r]& \sG}
\end{displaymath}
We then set $H(k) = \cofib (F(k-1) \to G(k))$.  It is clear that for $k \leq n$, $H(k)$ is a perfect object. Further, since we have the diagram
\begin{displaymath}
  \xymatrix{ F(k-1) \ar[d] \ar[r] & G(k) \ar[d] \ar[r] & H(k) \ar@{-->}[d] \\ F(k) \ar[r]& G(k+1) \ar[r]& H(k+1)} 
\end{displaymath}
we can imply the existence of the dotted arrow, making $H$ into a filtered object.   We will show that this choice of $H$ makes $\sH$ $n$-psuedo-coherent.  The first two conditions of Definition \ref{nspc} are clearly satisfied.  For the third, by playing with fiber/cofiber squares, 
$V_{H, k+1} \cong \cofib(V_{F,k} \to V_{G, k+1})$ (this follows easily in the dg-world from the explicit description of the cone). Since we are working locally, and $V_{F, k}$ is a projective module, the morphism $V_{F, k} \to V_{G, k+1} \cong 0$, thus $V_{H, k+1} \cong V_{F,k}[1] \oplus V_{G, k+1}$ and has Tor-amplitude of $-k-1$.  

Lastly, reindexing $F$: $F^\prime(i) := F(i-1)$ we have $\varinjlim F^\prime \cong \sF$ and a morphism of filtered objects $F^\prime \to G$.  Levelwise, $\sH(k) \cong \cofib(F^\prime(k) \to G(k))$.  By interchanging limits $\varinjlim H(k) \cong \varinjlim \cofib(F^\prime(k) \to G(k)) \cong \cofib (\varinjlim F(k) \to \varinjlim G(k)) \cong \cofib (\sF \to \sG)$, and thus the result.

{\it 2)} Let $\sH := \sF \oplus \sG$.  Working locally, we may assume that $\sH$ is strict $n$-psuedo-coherent with filtration $H$.  We first assume that $\sF, \sG \in \QC^{(-\infty, n]}(X)$.  We claim that setting
\begin{displaymath}
  F(i) = \begin{cases} H(i)  & i \leq n \\ \s{F} & i > n \end{cases}
\end{displaymath}
will make $\sF$ $n$-psuedo-coherent.  The morphism $F(n) \to F(n+1)$ is induced from $H(n) \to \sF\oplus\sG \to \sF$.  Clearly $\varinjlim F \cong \sF$ and Definition \ref{nspc}(1-3) hold.  Definition \ref{nspc}(4) holds since $\cofib(F(i-1) \to F(i)) \cong 0$ for $i > n+1$, while $\cofib (F(n) \to F(n+1)) \cong \cofib (H(n) \to \sF) \in \QC^{(-\infty, -1]}(X)$ (since $H(n) \to \sH$ is a $n$-quasi-isomorphism by Proposition \ref{conversion}).

To prove the general statement, we work by descending induction.  Working locally, we can assume $\sF \oplus \sG$ is bounded above.  Applying our above case, $\sF$ and $\sG$ are $N$-psuedo-coherent for some $N > n$. Assume that we have shown for $n < k \leq N$.  Let $F_k$ and $G_k$ be a choice of accompanying filtration for $\sF$ and $\sG$.  We have a triangle $F_k(k) \oplus G_k(k) \xrightarrow{i} \sF \oplus \sG$.  Since both of these objects are $(k-1)$-pseudo-coherent, by Proposition \ref{conversion} and Proposition \ref{portingTT}(1), $\cofib(i) \in \QC^{(-\infty, k-1)}(X)$, is $k-1$ psuedo-coherent, and decomposes as $\cofib(F_k(k) \to \sF) \oplus \cofib(G_k(k) \to \sG)$.  Applying our case above, we see $\cofib F_k \to \sF$ is $k-1$-psuedo-coherent.  Applying Proposition \ref{portingTT}(1) again, $\sF$ is $k-1$-psuedo-coherent, and induction gives the statement.  
\end{proof}

\begin{lemma}
\label{perfect_finite_tora}
Let $X$ be a geometric $D^-$ stack.
\begin{enumerate}
\item $\Perf(X) \subset \pscoh(X)$ and is defined by the property of locally finite Tor-amplitude.
\item $\sF \in \pscoh$ then $\tau^k(\sP)$ is $-k$-pseudo-coherent for all $k$. 
\end{enumerate}
\end{lemma}
\begin{proof}

{\it 1.}  Let $\sF$ be pseudo-coherent and locally finite Tor-amplitude. By refining the cover given for locally finite Tor-amplitude and pseudo-coherence, one can assume the existence of a cover $U : = \coprod \spec A_i$ with $\sF \big|_{\spec A_i}$ strict pseudo-coherent and finite Tor-amplitude.  The general case then will follow from the affine case. We must show that for $\sF \in \QC(A)$  strict pseudo-coherent and finite Tor-amplitude is necessary and sufficient for $\sF$ to be perfect.  In the case that $A$ is discrete, this is \cite[I 5.8.1]{SGA6}.  

Using standard homological methods, e.g., \cite[Proposition 2.2.12]{TT}, we can reduce to the case that $\sF$ has Tor-amplitude in $[0,0]$.  From \cite{TT} and Lemma \ref{pullback} , we know that $\pi_0(A)\otimes_A \sF$ is projective.  
Let $P$ be a projective $A$-module such that $\pi_0(P) \cong \pi_0(A)\otimes \sF$. The existence of this follows from the well known equivalence between projective $A$-modules and projective $\pi_0(A)$ modules.  Using the projectivity of $P$, we can lift the map $\pi_0(A)\otimes P \rightarrow \pi_0(A) \otimes \sF$ to a map $P \rightarrow \sF$.
We claim this is an equivalence.  For if not, then the homotopy cofiber will be non-trivial.  However, the Tor-amplitude of $\sF$ ensures that tensoring the homotopy fiber by $\pi_0(A)$ is trivial (since $\pi_0(A)\otimes \sF \cong \pi_0(A) \otimes P$).  However, $A$ is a unital algebra, thus this is impossible since then the identity would act trivially.         

For the converse direction, locally bounded Tor-amplitude and pseudo-coherence are local conditions, so we can assume that we are on an affine scheme.  The result then follows from finding a projective resolution.  For the discrete case, this is well known.  For the $\mE_\infty$ case, \cite{EKMM} can be used to obtain the result.

{\it 2.} If $\sF$ is strict pseudo-coherent, then let $F$ be a filtered object associated to $\sF$.  From Proposition \ref{conversion} the morphism $F(-m) \rightarrow \sF$ is a $m$-quasi-isomorphism.  Thus the filtered object
\begin{displaymath}
  F^k(i) = \begin{cases} F(i) & i \leq -k \\ \tau^k(\s{F}) & i > -k \end{cases}
\end{displaymath}
has the right Tor-amplitude properties.  Clearly $\varinjlim_{i} F^k(i) \cong \tau^k \sF$.
\end{proof}

\section{The biCartesian fibration $\QCF$}
\label{bicart}
We begin by implicitly choosing a universe $\mathbb{U}$ for which to work in and let $k$ be a commutative simplicial $\mZ$-algebra. The category of simplicial $\mZ$-modules,  denoted $\mZ\textrm{-mod}_{\Delta}$, is a simplicial symmetric monodial model category.  It is naturally identified with the full subcategory of connective objects in $Sp^\Sigma(s\s{A}b)$ (the symmetric monodial model category of symmetric sequences of simplicial $\mZ$-modules).   By \cite[Proposition 4.3]{shipley} this latter category is strongly monodial Quillen equivalent to $H\mZ$-mod, the symmetric monodial simplicial category of symmetric $H\mZ$-module spectra.  The functor $Sp^\Sigma(s\s{A}b) \to H\mZ$-mod is induced by the forgetful morphism $Ab \to Set_\ast$.  Clearly this preserves the smash product, so one gets a morphism between the model categories of commutative algebra objects: $\mr{CAlg}(Sp^{\Sigma}(\s{A}b)) \to \mr{CAlg}(H\mZ\textrm{-mod})$ (this is not a Quillen equivalence).  

Applying \cite[Theorem 4.4.4.7]{higher_algebra}\footnote{similar to Example 4.4.4.9, loc. cit.} one has an equivalence of $\infty$-categories:
\begin{displaymath}
N((\mr{CAlg}(H\mZ\textrm{-mod}))^{cf}) \cong \mr{CAlg}(N(H\mZ\textrm{-mod}^{cf})^\otimes)
\end{displaymath}
where $N(H\mZ\textrm{-mod}^{cf})^\otimes$ is the natural symmetric monodial $\infty$-category associated to $H\mZ$-mod and $\mr{CAlg}(N(H\mZ\textrm{-mod}^{cf})^\otimes)$ is the $\infty$-category of ring objects associated to the $\infty$-operad $N(\mr{Fin}_\ast)$, i.e the $\infty$-category of $\mathbb{E}_\infty$-$H\mZ$-algebras.  Since 
\begin{displaymath}
\mZ\textrm{-Alg}_{\Delta} := N(\mr{CAlg}(\mZ\textrm{-mod}_{\Delta})^{cf}) \subset  N(\mr{CAlg}(Sp^{\Sigma}(s\s{A}b))^{cf})
\end{displaymath}
by definition (using the notation in \S \ref{parlance}), composition results in a functor from the $\infty$-category of simplicial commutative $\mZ$-algebras to the $\infty$-category of $\mathbb{E}_\infty$-$H\mZ$-algebras.


If we denote the (large) category of symmetric monodial $\infty$-categories by $\widehat{Cat_\infty^\otimes}$, then using  \cite[Section 2.7]{DAGVIII} and \cite[Theorem 4.4.3.1]{higher_algebra}, one defines a functor
\begin{displaymath}
\mr{Mod}: \mr{CAlg}(N(H\mZ\textrm{-mod})^{cf,\otimes}) \to \widehat{Cat_{\infty}^{\otimes}}
\end{displaymath}
which on objects is the assignment $A \to Mod(A)$.  We have a composable sequence of functors 
\begin{align*}
  \kAlg \to \mZ\textrm{-Alg}_{\Delta} \to \mr{CAlg}(N(H\mZ\textrm{-mod})^{cf,\otimes}) \to \widehat{Cat_{\infty}^{\otimes}}
\end{align*}
The composite assigns to $A \in \kAlg$ its symmetric monodial stable $\infty$-category of $A$-modules.  For reasons similar to \cite[Section 2.7]{DAGVIII}, we Kan extend this functor to a limit preserving 
\begin{align*}
  \widehat{QC}^{cart}(k): \mr{Fun}(\kAlg, \s{S})^{\mr{op}} \to \widehat{Cat_{\infty}}
\end{align*}
which associates to a functor $X$, its symmetric monodial $\infty$-category of quasi-coherent sheaves (technically, the monodial structure is extra data that can be added).  To a morphism $f: X \to Y$, the resulting morphism $f^\ast$ is a symmetric monodial functor. 
Passing to the associated Cartesian fibrations and restricting to stacks, we get a Cartesian fibration $\widehat \QC^{cart}(k) \to dSt_{k}$.  
  As mentioned in \cite[Section 2.7]{DAGVIII}, the fibers of this fibration are naturally symmetric monodial $\infty$-categories and the Cartesian lifts of morphisms are symmetric monodial.


Now, let $X \in dSt_{k}$ be a geometric stack.  Using the Cartesian monodial structure on $dSt_{k}$ we obtain a natural functor $dAff_k \xrightarrow{X \times -} dSt_k$ which on $0$-simplices assigns $\spec A \to X \times \spec A$.  Formally, one considers the composition of functors $dAff  \to dAff \times dSt_{k}$ given by $\times X$.  This latter category is equivalent to a subcategory of $(dSt_k^\otimes)_{2}$.  Composition with the the active morphism $(dSt_k^\otimes)_{2} \xrightarrow{\phi_{2,1}} (dSt_k^\otimes)_1$ is the desired functor.  We let $\widetilde{\QC}^{cart}(X)$ be the pullback Cartesian fibration  of $\widehat{\QC}^{cart}(k)$ under this functor.  

We are left with promoting this Cartesian fibration to a biCartesian fibration, i.e., finding a biCartesian fibration $\QCF$ equivalent (as Cartesian fibrations) to $\widetilde{\QC}^{cart}(X)$.  By \cite[Proposition 6.2.3.17]{higher_algebra} it is enough to show that for any $f: A \to B$ of $k$-algebras, $f^\ast: \QC(X \times \spec A) \to \QC(X \times \spec B)$ has a right adjoint.  However, this is clear since $X \times \spec B \to X \times \spec A$ is an affine morphism, i.e., it follows from the affine case.   

Lastly, it will be more convenient to work over $\kAlg$ than $dAff_{k}$.  Since there is a natural correspondence between Cartesian (coCartesian) fibrations over S and coCartesian (Cartesian, respectively) fibrations over $S^{op}$, we will generally think of $\QCF$ as a biCartesian fibration over $\kAlg$.

\bibliography{./biblio}
\bibliographystyle{halpha}
\end{document}